\documentclass[11pt,reqno]{amsart}
\usepackage{fullpage,times,graphicx,amssymb,amsmath,amsfonts,bbm,psfrag,xcolor, thmtools, thm-restate}
\definecolor{ddarkbrown}{rgb}{0.5,0.2,0.05} \definecolor{bbluegray}{rgb}{0.05,0,0.5}
\definecolor{puorange}{rgb}{0.80,0.20,0}
\usepackage[colorlinks,citecolor=bbluegray,linkcolor=ddarkbrown,urlcolor=blue,breaklinks]{hyperref}
\usepackage[square,comma]{natbib}
\definecolor{ddarkbrown}{rgb}{0.5,0.2,0.05} \definecolor{bbluegray}{rgb}{0.05,0,0.5}

\newtheorem{theorem}{Theorem}[section]
\newtheorem{proposition}[theorem]{Proposition}
\newtheorem{definition}[theorem]{Definition}
\newtheorem{lemma}[theorem]{Lemma}
\newtheorem{corollary}[theorem]{Corollary}

\newcommand{\BEAS}{\begin{eqnarray*}}
\newcommand{\EEAS}{\end{eqnarray*}}
\newcommand{\BEA}{\begin{eqnarray}}
\newcommand{\EEA}{\end{eqnarray}}

\newcommand{\BEQ}{\begin{equation}}
\newcommand{\EEQ}{\end{equation}}
\newcommand{\BIT}{\begin{itemize}}
\newcommand{\EIT}{\end{itemize}}
\newcommand{\BNUM}{\begin{enumerate}}
\newcommand{\ENUM}{\end{enumerate}}

\newcommand{\BA}{\begin{array}}
\newcommand{\EA}{\end{array}}

\newcommand{\reals}{{\mathbb R}}

\newcommand{\argmin}{\mathop{\rm argmin}}

\newcommand{\dom}{\mathop{\bf dom}}

\usepackage{algorithm,algorithmicx,algpseudocode}

\algnewcommand{\IIf}[1]{\State\algorithmicif\ #1\ \algorithmicthen}
\algnewcommand{\EndIIf}{\unskip\ \algorithmicend\ \algorithmicif}

\usepackage{auto-pst-pdf}

\title{Sharpness, Restart and Acceleration}

\author{Vincent Roulet} 
\address{Department of Statistics,\vskip 0ex University of Washington, Seattle, USA.}
\email{vroulet@uw.edu}

\author{Alexandre d'Aspremont}
\address{CNRS \& D.I., UMR 8548,\vskip 0ex
\'Ecole Normale Sup\'erieure, Paris, France.}
\email{aspremon@ens.fr}

\begin{document} 
	\maketitle
	
	\begin{abstract}
		The {\L}ojasiewicz inequality shows that sharpness bounds on the minimum of convex optimization problems hold almost generically.  Sharpness directly controls the performance of restart schemes, as observed by \citet{Nemi85}. The constants quantifying these sharpness bounds are of course unobservable, but we show that optimal restart strategies are robust, in the sense that, in some important cases, finding the best restart scheme only requires a log scale grid search. Overall then, restart schemes generically accelerate accelerated first-order methods.
	\end{abstract}

	\section*{Introduction}
	We study\footnote{A subset of these results appeared at the NIPS 2017 conference under the same title.} convex optimization problems of the form
\BEQ\label{eq:optim_pb_intro}\tag{P}
\BA{ll}
\mbox{minimize} & f(x)
\EA
\EEQ
where $f$ is a convex
function defined on $\reals^n$. The complexity of these problems using first order methods is usually controlled by smoothness assumptions on $f$ such as Lipschitz continuity of its gradient.
Additional assumptions such as strong or uniform convexity provide respectively linear and faster polynomial rates of convergence \citep{Nest13a, Ioud14}. However, these assumptions are often too restrictive to be applicable. 
Here, we make a much more generic assumption that describes the growth of the function around its minimizers using constants $\mu>0$ and $r\geq 1$ such that
\BEQ\label{eq:KL}\tag{Loja}
\frac{\mu}{r} d(x,X^*)^r \leq f(x)-f^*, \quad \mbox{for every $x\in K$},
\EEQ
where $f^*$ is the minimum of $f$, $K\supset X^*$ is a given set and $d(x,X^*) = \min_{y\in X^*} \|x-y\|_2$ is the Euclidean distance from $x$ to the set  $X^*$ of minimizers of~$f$. This defines a \emph{lower bound} on the function around its minimizers and quantifies the sharpness of the minimum. We exploit this property using restart schemes on classical convex optimization algorithms.

The sharpness assumption~\eqref{eq:KL} is also known as a H\"olderian error bound on the distance to the set of minimizers. \citet{Hoff52} first introduced error bounds to study systems of linear inequalities. Natural extensions were then developed for convex optimization by \citet{Robi75,Manga85,Ausl88}, notably through the concept of sharp minimum \citep{Burk93,Burk02a}. But the most striking result in this vein is due to \citet{Loja63,Loja93} who proved that inequality~\eqref{eq:KL} holds generically for real analytic and subanalytic functions. This result has then been extended to non-smooth subanalytic convex functions by \citet{Bolt07}. Overall then, condition~\eqref{eq:KL} essentially measures the sharpness of minimizers, and holds generically. On the other hand, this inequality is purely implicit as $r$ or $\mu$ are neither observed nor known a priori, and deriving adaptive schemes is thus crucial to ensure practical relevance.

{\L}ojasiewicz inequalities either in the form of~\eqref{eq:KL} or as gradient dominated properties \citep{Poly63} led to new convergence results for composite problems and for alternating or splitting methods \citep{Atto10,Bolt14,Fran15,Kari16}. Here we use this inequality to produce accelerated rates for restart schemes.

Restart schemes have already been studied for strongly or uniformly convex functions in e.g. \citep{Nemi85,Nest13b,Ioud14,Lin14}. In particular, \citet{Nemi85} link a ``strict minimum'' condition akin to~\eqref{eq:KL} with faster convergence rates using restart schemes which form the basis of our results, but they do not study the cost of adaptation and do not tackle the non-smooth case. In a similar spirit, weaker versions of this strict minimum condition were used more recently to study the performance of restart schemes in \citep{Rene14,Freu15,Roul15}. 

The fundamental question regarding restart schemes is to define when to restart. Several heuristics have been presented that used some criterion on the iterates to  restart the accelerated algorithm and speed up convergence~\citep{Odon15,Su14,Gise14}. However, they did not theoretically establish improved  complexity bounds. The robustness of restart schemes was also studied by \citet{Ferc16} for quadratic error bounds, i.e.~\eqref{eq:KL} with $r=2$, satisfied by the LASSO problem for example. \citet{Ferc17} recently extended this work to produce adaptive restarts with theoretical guarantees of optimal performance, again for quadratic error bounds. In the same vein, \citet{Liu17} presented adaptive accelerated methods given H\"olderian error bounds, but their results are not adaptive to the exponent of the error bound. The references above focus on smooth problems, but error bounds appear also for non-smooth ones, with \citet{Gilp12} proving for example linear convergence of restart schemes in bilinear matrix games where the minimum is sharp, i.e.~\eqref{eq:KL} with $r=1$. Recently \citet{Rene18} presented simple generic schemes inspired by an early draft of this work, and provide adaptive schemes in all regimes (not only the smooth case).

Our contribution here is to derive optimal scheduled restart schemes for general convex optimization problems on smooth, non-smooth or H\"older smooth functions satisfying a sharpness assumption. We then show that for smooth functions these schemes can be made adaptive with nearly optimal complexity (up to a squared log term) for a wide array of sharpness assumptions. We also analyze restart schemes based on a sufficient decrease of the primal gap, when the optimal value of the problem is known. In that case, restart schemes are shown to be optimal without requiring a log scale grid search on the parameters. Our proofs only rely on having access to the convergence bound of an accelerated method, therefore our results are directly extended to the non-Euclidean case with composite objective and to non-smooth functions that can be smoothed.
	
	\section{Regularity Assumptions}
	\subsection{Smoothness}
Convex optimization problems~\eqref{eq:optim_pb_intro} are generally divided in two classes: smooth problems, for which $f$ has Lipschitz continuous gradients, and non-smooth problems for which $f$ is not differentiable. 
Following \citet{Nest15}, we use a unified framework that extends the definition of H\"older smooth functions.
\begin{definition}\label{def:holder}
A function  $f$ is $s$-smooth for given $1\leq s\leq 2$ if there exists a constant $L>0$ such that
\BEQ\label{eq:smooth}\tag{H\"older}
\|\nabla f(x)-\nabla f(y) \|_2 \leq L \|x-y\|_2^{s-1}, \quad  \mbox{for all $x,y \in \dom f$}
\EEQ
and any subgradients $\nabla f(x) \in \partial f(x), \nabla f(y)  \in \partial f(y)$ of $f$ at $x, y$ respectively.
We write $\mathcal{H}_{s, L}$ the set of $s$-smooth functions with parameter $L$.
\end{definition}

For $s=2$, we retrieve the classical definition of smoothness \citep{Nest13a}. For $s=1$ we get a classical assumption made in non-smooth convex optimization, i.e. that sub-gradients of the function are bounded. For $s\in]1, 2[$ we get the definition of H\"older smooth functions.
We generalize our results for functions smooth with respect to a non-Euclidean norm in Section~\ref{sec:compobreg}.

\subsection{Sharpness, Error Bounds}\label{sec:Loja}
We study convex optimization problems whose objective satisfies a growth condition as defined below.
\begin{definition}
	A function $f$ satisfies a {\L}ojasiewicz growth condition on a set $K$ if there exist  constants $r\geq 1$, $\mu>0$, such that
	\BEQ\tag{Loja}
	\frac{\mu}{r} d(x,X^*)^r \leq f(x)-f^*, \quad \mbox{for every $x\in K$},
	\EEQ
	where $f^*$ is the minimum of $f$, $d(x,X^*) = \min_{y\in X^*} \|x-y\|_2$ is the Euclidean distance from $x$ to the set  $X^*$ of minimizers of~$f$.
	We write $\mathcal{L}_{r, \mu}(K)$ the set of functions satisfying a {\L}ojasiewicz growth condition on a set $K$ with  parameters $r\geq 1$, $\mu>0$.
\end{definition}
Condition~\eqref{eq:KL} holds almost generically, and is notably satisfied by analytic and subanalytic functions (see~\citep{Bolt15} for more details). However, the proof (see e.g. \citet[Theorem 6.4]{Bier88}) uses topological arguments that are far from constructive. Hence, outside of some particular cases (e.g. strong convexity), we cannot assume that the constants in~\eqref{eq:KL} are known, even approximately.

Error bounds are directly related to a {\L}ojasiewicz inequality bounding the magnitude of the gradient~\citep{Bolt15}. These properties underlie many recent results in optimization \citep{Atto10,Fran15,Bolt14}. Here, the sharpness condition in~\eqref{eq:KL} allows us to accelerate convex optimization algorithms using restart schemes. 

Our analysis relies on the condition that~\eqref{eq:KL} is satisfied for any output of the algorithms we restart. By enforcing monotonicity of the objective values produced by those algorithms, this reduces to assume that~\eqref{eq:KL} is satisfied on sublevel sets of the objective. 

\subsection{Sharpness and Smoothness} \label{sec:sharp_smooth}
Given a convex function $f \in \mathcal{H}_{s, L}$, by using its Taylor expansion and the smoothness property, we get
$
f(x) \leq f^*+ \frac{L}{s}\|x-y\|_2^s,
$
for $x \in \dom f$ and $y \in X^*$. Setting $y$ to be the projection of $x$ onto $X^*$, this yields the following {\em upper bound} on suboptimality
\BEQ\label{eq:lb-s}
f(x) - f^* \leq \frac{L}{s} d(x,X^*)^s.
\EEQ
Now, assume moreover that $f \in \mathcal{L}_{r, \mu}(K)$ for a  given set $K$ such that $ X^* \subset K \subset \dom f$.
Combining~\eqref{eq:lb-s} and~\eqref{eq:KL} leads to
\[
\frac{s\mu}{rL} \leq d(x,X^*)^{s-r},
\]
for every $x \in K \setminus X^*$. This means that necessarily $s\leq r$ by taking $x$ close enough to $X^*$.
Moreover if $s<r$, the set $K$ must satisfy $\sup_{x \in K} d(x, X^*) <+\infty$.

For the following, we define
\BEQ\label{eq:kappa_tau}
\kappa \triangleq \frac{L^\frac{2}{s}}{\mu^\frac{2}{r}} \qquad \mbox{and} \qquad  \tau \triangleq 1-\frac{s}{r} \in [0,1)
\EEQ
a generalized condition number for the function $f$ and a condition number based on the ratio of powers in inequalities~\eqref{eq:smooth} and~\eqref{eq:KL}, respectively. Note that if $r=s=2$, $\kappa$ matches the classical condition number of the function.
	
	\section{Scheduled Restarts for Smooth Convex Problems}\label{sec:scheduled}
	In this section, we seek to solve \eqref{eq:optim_pb_intro} assuming that the function $f$ is smooth, i.e. satisfies \eqref{eq:smooth} with $s=2$ and $L>0$. Without further assumptions on $f$, an optimal algorithm to solve the smooth convex optimization problem~\eqref{eq:optim_pb_intro} is Nesterov's accelerated gradient method \citep{Nest83}. Given an initial point $x_0$, this algorithm outputs, after $t$ iterations, a point
\BEQ\label{eq:algobound_s=2}
x = \mathcal{A}(x_0,t) \qquad \mbox{such that} \qquad f(x) - f^* \leq  \frac{cL}{t^2} d(x_0,X^* )^2,
\EEQ
where $c>0$ is a universal constant (whose value will be allowed to vary in what follows, with $c=4$ here). The accelerated algorithm can be enforced to output solutions whose objective decays monotonically as detailed in Appendix~\ref{sec:algos}. Consequently, if $f$ satisfies a {\L}ojasiewicz growth condition on the initial sub-level set $K = \{x: f(x) \leq f(x_0)\}$, then it is satisfied for any point output by the algorithm. 

Note that the arguments that we develop below are not specific to the algorithm of~\citet{Nest83} and would apply to any method satisfying the complexity bound~\eqref{eq:algobound_s=2} as shown for example in Section~\ref{sec:compobreg} that generalizes the results to the non-Euclidean setting.
We now describe a restart scheme exploiting the extra regularity~\eqref{eq:KL} to improve the computational complexity of solving problem~\eqref{eq:optim_pb_intro} using accelerated methods.

\subsection{Scheduled restarts}\label{ssec:scheduled_restarts}
Here, we schedule the number of iterations $t_k$ made by the accelerated gradient algorithm between restarts, with $t_k$ being the number of (inner) iterations at the $k^\mathrm{th}$ algorithm run (outer iteration). Our scheme is described in Algorithm~\ref{algo:scheduled_s=2} below.
\begin{algorithm}[H]
	\caption{Scheduled restarts for smooth convex minimization \label{algo:scheduled_s=2}}
	\begin{algorithmic}
		\State{\textbf{Inputs :} $x_0\in\reals^n$ and a sequence $t_k$ for $k=1,\ldots,R$.}
		\For{$k=1,\ldots,R $}
		\State
		\vspace*{-0.3cm}
		\[
		x_k := \mathcal{A}(x_{k-1},t_k) 
		\]
		\vspace*{-0.5cm}
		\EndFor
		\State{\textbf{Output :} $\hat{x} := x_R$}
	\end{algorithmic}
\end{algorithm}
The analysis of this scheme and the following ones rely on two steps. We first choose schedules that ensure linear convergence of the objective values $f(x_k)$ w.r.t. $k$ at a given rate. We then adjust this linear rate to minimize complexity, i.e. the total number of inner iterations.
We begin with a technical lemma which assumes linear convergence holds, and connects the growth of~$t_k$, the precision reached and the total number of inner iterations $N$.
\begin{lemma}\label{lem:total_number}
	Let $x_k$ be a sequence whose $k\textsuperscript{th}$ iterate is generated from the previous one by an algorithm that runs $t_k$ iterations and write $N = \sum_{k=1}^R t_k$ the total number of iterations to output a point $x_R$. Suppose setting $t_k = Ce^{\alpha k}$, $(k=1,\ldots,R)$ for some $C>0$ and $\alpha\geq 0$ ensures that the outer iterations satisfy
	\BEQ \label{eq:conv_obj}
	f(x_k) -f^* \leq \nu e^{-\gamma k},
	\EEQ
	for all $k\geq 0$ where $\nu\geq 0$ and $\gamma\geq 0$. Then, precision at the output is given by,
	\[
	f(x_R) -f^* \leq \nu\exp(-\gamma  N/C),\quad \mbox{when $\alpha=0$,}
	\]
	and
	\[
	f(x_R) -f^* \leq \frac{\nu}{\left(\alpha e^{-\alpha} C^{-1}N +1\right)^\frac{\gamma}{\alpha}},\quad \mbox{when $\alpha>0$.} 
	\]
\end{lemma}
\begin{proof}
	When $\alpha=0$, $N = RC$, and inserting this in \eqref{eq:conv_obj} at the last point $x_R$ yields the desired result. On the other hand, when $\alpha>0$, we have $N = \sum_{k=1}^R t_k = C e^{\alpha} \frac{e^{\alpha R}-1}{e^{\alpha}-1}$, which gives
	$
	R = \log\left(\frac{e^{\alpha}-1}{e^{\alpha} C}N+1\right)/ \alpha.
	$
	Inserting this in \eqref{eq:conv_obj} at the last point, we get
	\[
	f(x_R) - f^*  \leq  \nu \exp \left(-\frac{\gamma}{\alpha} \log\left(\frac{e^{\alpha}-1}{e^{\alpha} C}N+1\right)\right) 
	\leq  \frac{\nu}{\left(\alpha e^{-\alpha} C^{-1}N+1\right)^\frac{\gamma}{\alpha}},
	\]
	where we used $e^x-1\geq x$. This yields the second part of the result.
\end{proof}
The last approximation in the case $\alpha>0$ simplifies the analysis that follows, without significantly affecting the bounds. We also show in Appendix~\ref{sec:total_number_approx} that using integer values $\tilde{t}_k = \lceil t_k \rceil$ does not significantly affect the bounds above. 

We now analyze restart schedules $t_k$ that ensure linear convergence. Our choice of $t_k$ will heavily depend on the ratio between $r$ and $s$ (with $s=2$ for smooth functions here), measured by $\tau=1-s/r$ defined in~\eqref{eq:kappa_tau}. Below, we show that if $\tau=0$, a constant schedule is sufficient to ensure linear convergence. When $\tau>0$, we need a geometrically increasing number of iterations for each cycle. 
\begin{proposition}\label{th:sched}
	Let $f$ be a convex function and $x_0 \in \dom f$. Denote $K = \{x: f(x) \leq f(x_0) \}$ and assume that $f \in \mathcal{H}_{2, L} \cap \mathcal{L}_{r, \mu}(K)$.
	Run Algorithm~\ref{algo:scheduled_s=2} from $x_0$ with iteration schedule $t_k=C^*_{\kappa,\tau}e^{\tau k}$, for $k=1,\ldots,R$,
	where 
	\BEQ\label{def:C}
	C^*_{\kappa,\tau} \triangleq e^{1-\tau}(c\kappa)^\frac{1}{2}(f(x_0)-f^*)^{-\frac{\tau}{2}},
	\EEQ
	with $\kappa$ and $\tau$ defined in~\eqref{eq:kappa_tau} and $c = 4e^{2/e}$ here. The precision reached at the last point $\hat{x}$ is given by,
	\begin{align}\label{eq:conv_s=2_r=2_opt}
	f(\hat{x}) - f^* \leq \exp\left(-2e^{-1}(c\kappa)^{-\frac{1}{2}} N\right)(f(x_0)-f^*) = O\left(\exp(-\kappa^{-\frac{1}{2}} N)\right),  \quad \mbox{when $\tau =0$,}
	\end{align}
	while,
	\begin{align}\label{eq:conv_s=2_r>2_opt}
	f(\hat{x}) -f^* \leq \frac{f(x_0)-f^*}{\left(\tau e^{-1}(f(x_0)-f^*)^{\frac{\tau}{2}}  (c\kappa)^{-\frac{1}{2}}N +1\right)^\frac{2}{\tau}} = O\left(N^{-\frac{2}{\tau}}\right), \quad \mbox{when $\tau >0$,}
	\end{align}
	where $N=\sum_{k=1}^R t_k$ is the total number of iterations.
\end{proposition} 
\begin{proof} Our strategy is to choose $t_k$ such that the objective is linearly decreasing, i.e. 
	\BEQ\label{eq:conv_delta_f_s=2}
	f\left(x_k\right)-f^* \leq e^{-\gamma k} (f(x_0)-f^*),
	\EEQ
	for some $\gamma\geq 0$ depending on the choice of $t_k$. 
	This directly holds for $k=0$ and any $\gamma\geq 0$. Combining~\eqref{eq:KL} with the complexity bound in \eqref{eq:algobound_s=2}, we get 
	\BEQ\label{eq:shrp-complx}
	f\left(x_k\right)-f^* \leq  \frac{c \kappa}{t_k^2}(f\left(x_{k-1}\right)-f^*)^\frac{2}{r},
	\EEQ
	where $c  = 4 e^{2/e} $ using that $r^{2/r} \leq e^{2/e}$.
	Assuming recursively that \eqref{eq:conv_delta_f_s=2} is satisfied at iteration $k-1$ for a given $\gamma$, we have 
	\[
	f\left(x_k\right)-f^* \leq \frac{c \kappa  e^{-\gamma\frac{2}{r}(k-1)}}{t_k^2}(f(x_0)-f^*)^\frac{2}{r},
	\]
	and to ensure \eqref{eq:conv_delta_f_s=2} at iteration $k$, we impose
	\[
	\frac{c\kappa e^{-\gamma \frac{2}{r}(k-1)}}{t_k^2}(f(x_0)-f^*)^\frac{2}{r} \leq e^{-\gamma k}(f(x_0)-f^*).
	\] 
	Rearranging terms in this last inequality, using $\tau$ defined in \eqref{eq:kappa_tau}, we get 
	\BEQ\label{eq:cond_for_conv_s_2}
	t_k \geq e^{ \frac{\gamma(1-\tau)}{2}} (c \kappa)^{\frac{1}{2}}  (f(x_0)-f^*)^{-\frac{\tau}{2}} e^{\frac{\tau \gamma }{2}k}.
	\EEQ
	For a given $\gamma\geq 0$, we can set $t_k = C e^{\alpha k}$ where 
	\BEQ
	C =  e^{ \frac{\gamma(1-\tau)}{2}} (c \kappa)^{\frac{1}{2}} (f(x_0)-f^*)^{-\frac{\tau}{2}} \qquad \mbox{and} \qquad
	\alpha =  {\tau \gamma }/{2},   
	\EEQ
	and Lemma~\ref{lem:total_number} then yields, 
	\[
	f(\hat{x}) - f^* \leq  \exp\left(-\gamma e^{-\frac{\gamma}{2}}(c\kappa)^{-\frac{1}{2}} N\right)(f(x_0)-f^*),
	\]
	when $\tau =0$, while
	\[
	f(\hat{x})-f^* \leq  \frac{f(x_0)-f^*}
	{\left(\frac{\tau }{2}\gamma e^{-\frac{\gamma}{2}} (c\kappa)^{-\frac{1}{2}}(f(x_0)-f^*)^{\frac{\tau}{2}}N + 1
		\right)^\frac{2}{\tau}},
	\]
	when $\tau >0$. These bounds are minimal for $\gamma = 2$, which yields the desired result.
\end{proof}
When $\tau=0$, bound \eqref{eq:conv_s=2_r=2_opt} matches the classical complexity bound for smooth strongly convex functions~\citep{Nest13a}. When $\tau>0$ on the other hand, bound \eqref{eq:conv_s=2_r>2_opt} highlights a {\em faster convergence rate than accelerated gradient methods.} The sharper the function (i.e. the closer $r$ is to $2$), the faster the convergence. This matches the lower bounds for optimizing smooth and sharp functions functions  up to constant factors \citep[Eq. 1.21]{Nemi85}.
Also, setting $t_k = C^*_{\kappa,\tau}e^{\tau k}$ yields continuous bounds on precision, i.e. when $\tau \rightarrow 0$, bound \eqref{eq:conv_s=2_r>2_opt} converges to bound \eqref{eq:conv_s=2_r=2_opt}, which also shows that for $\tau$ near zero, constant restart schemes are almost optimal. 

Note that for $N\leq C^*_{\kappa,\tau}$, the bounds~\eqref{eq:conv_s=2_r=2_opt},~\eqref{eq:conv_s=2_r>2_opt} are not informative. Precisely, the lower bounds for this problem as presented in~\citep[Eq. 1.21]{Nemi85} are not informative for small $N$. In that case, the optimal rate is given by the accelerated scheme and consequently by Algorithm~\ref{algo:scheduled_s=2} before the first restart. 

\subsection{Adaptive scheduled restart}\label{ss:adapt}
The previous restart schedules depend on the sharpness parameters $(r,\mu)$ in~\eqref{eq:KL}. In general of course, these values are neither observed nor known a priori. Making the restart scheme  adaptive is thus crucial to its practical performance. Fortunately, we show below that a simple logarithmic grid search on these parameters is enough to guarantee nearly optimal performance.

We begin with the following Proposition that stems from the proof of Proposition~\ref{th:sched}.
\begin{proposition}\label{cor:rates}
	Let $f$ be a convex function and $x_0 \in \dom f$. Denote $K = \{x: f(x) \leq f(x_0) \}$ and assume that $f \in \mathcal{H}_{2, L} \cap \mathcal{L}_{r, \mu}(K)$.
	Run Algorithm~\ref{algo:scheduled_s=2} from $x_0$ with general schedules of the form 
	\[
	\left\{\BA{lll}
	t_k = C & \mbox{if $\tau = 0 $}, \\
	t_k = C e^{\alpha k} & \mbox{if $\tau > 0 $}.
	\EA\right.
	\]
	If $\tau =0$ and $C \geq C^*_{\kappa,0}$, then 
	\BEQ\label{eq:conv_s_2_gen_schedule_tau=0}
	f(\hat{x})-f^* \leq \left(\frac{c\kappa}{C^2}\right)^\frac{N}{C} (f(x_0)-f^*),
	\EEQ
	while, if $\tau >0$ and $C \geq C(\alpha)$, then
	\BEQ\label{eq:conv_s_2_gen_schedule_tau>0}
	f(\hat{x})-f^* \leq \frac{ f(x_0)-f^*}{\left(\alpha e^{-\alpha} C^{-1}N+1\right)^\frac{2}{\tau}},
	\EEQ
	where 
	\BEQ\label{eq:C_alpha}
	C(\alpha) \triangleq e^{ \frac{\alpha(1-\tau)}{\tau}} (c \kappa)^{\frac{1}{2}}  (f(x_0)-f^*)^{-\frac{\tau}{2}},
	\EEQ
	and $N=\sum_{k=1}^R t_k$ is the total number of iterations.
\end{proposition}
\begin{proof}
	Given general schedules of the form
	\[
	\left\{\BA{lll}
	t_k = C & \mbox{if $\tau = 0 $}, \\
	t_k = C e^{\alpha k} & \mbox{if $\tau > 0 $},
	\EA\right.
	\]
	the best value of $\gamma$ satisfying condition \eqref{eq:cond_for_conv_s_2} for any $k\geq 0$ in Proposition~\ref{th:sched} is given by
	\[
	\left\{
	\BA{ll}
	\gamma = \log\left(\frac{C^2}{c\kappa}\right) & \mbox{if $\tau = 0 $ and $C \geq C^*_{\kappa,0}$,} \\
	\gamma = \frac{2 \alpha}{\tau}  & \mbox{if $\tau > 0 $ and $C \geq C(\alpha)$.}
	\EA\right.
	\]
	As in Proposition~\ref{th:sched}, plugging these values into the bounds of Lemma~\ref{lem:total_number} yields the desired result.
\end{proof}
We run several schemes with a fixed number of inner iterations $N$ to perform a log-scale grid search on $\tau$ and $\kappa$. We define these schemes as follows.
\BEQ\label{algo:universal_s=2}
\left\{\BA{l}
\mathcal{S}_{i,0}:\mbox{Restart Algorithm~\ref{algo:scheduled_s=2} with $t_k = C_i$}, \\
\mathcal{S}_{i,j}:\mbox{Restart Algorithm~\ref{algo:scheduled_s=2} with $t_k = C_i e^{\tau_j k}  $,} \\
\EA
\right.
\EEQ
where $C_i = 2^i$ and $\tau_j = 2^{-j}$. We stop each of these schemes when the total number of its inner iterations has exceeded $N$, i.e. at the smallest $R$ such that $\sum_{k=1}^R t_k \geq N$. The size of the grid search in $C_i$ is naturally bounded as we cannot restart the algorithm after more than $N$ total inner iterations, so $i\in [1,\ldots,\lfloor \log_2 N \rfloor]$. We also show that when $\tau$ is smaller than $1/N$, a constant schedule where $t_k=C$ performs as well as the optimal geometrically increasing schedule where $t_k=C^*_{\kappa,\tau}e^{\tau k}$. This crucially means we can also choose $j \in [1,\ldots,\lceil \log_2 N \rceil ]$, hence limiting the cost of the grid search.

The following proposition details the convergence of this grid-search, using the same notations as in Proposition~\ref{th:sched}.
As observed at the end of Section~\ref{ssec:scheduled_restarts}, the optimal bounds~\eqref{eq:conv_s=2_r=2_opt},~\eqref{eq:conv_s=2_r>2_opt} are only informative after a sufficient number of iterations, which is why we analyze the adaptive scheme only for a number of iterations $N \geq 2 C^*_{\kappa, \tau}$. To get optimal bounds in all regimes, it suffices to run an additional non-restarted algorithm that will also capture the best rate in the case $N < 2 C^*_{\kappa, \tau}$. 
\begin{proposition}\label{prop:grid_search}
	Let $f$ be a convex function and $x_0 \in \dom f$. Denote $K = \{x: f(x) \leq f(x_0) \}$, assume that $f \in \mathcal{H}_{2, L} \cap \mathcal{L}_{r, \mu}(K)$ and denote by $N\geq 2C^*_{\kappa,\tau}$ a given number of iterations.

	Run schemes $\mathcal{S}_{i,j}$ defined in \eqref{algo:universal_s=2} to solve \eqref{eq:optim_pb_intro} for $i\in [1,\ldots,\lfloor \log_2 N \rfloor]$ and $j \in [0,\ldots,\lceil \log_2 N \rceil]$, stopping each time after $N$ total inner algorithm iterations, i.e. for $R$ such that $\sum_{k=1}^R t_k \geq N$.
	
	If $\tau = 0$, there exists $i \in [1,\ldots,\lfloor \log_2 N \rfloor ]$ such that the scheme $\mathcal{S}_{i,0}$ achieves a precision given by
	\[
	f(\hat{x}) -f^* \leq \exp\left(-e^{-1}(c\kappa)^{-\frac{1}{2}} N\right)(f(x_0)-f^*).
	\]
	
	If $\tau > 0$, there exist $i \in [1,\ldots,\lfloor \log_2 N \rfloor ]$ and $j \in [0,\ldots,\lceil \log_2 N \rceil]$ such that the scheme $\mathcal{S}_{i,j}$ achieves a precision given by
	\[
	f(\hat{x}) -f^* \leq  \frac{f(x_0)-f^*}{\left(\tau e^{-1}(c\kappa)^{-\frac{1}{2}} (f(x_0)-f^*)^{\frac{\tau}{2}}  (N-1)/4 +1\right)^\frac{2}{\tau}}.
	\]
	Overall, running the logarithmic grid search has a complexity $(\log_2 N )^2$ times higher than running $N$ iterations using the optimal (oracle) scheme.
\end{proposition}
\begin{proof}
	Denote $R$ the number of restarts of a scheme $S_{ij}$, we have for $j=0$, $R = \lceil N/C_i \rceil$ and for $j\neq 0$, $R = \lceil \log\left(\frac{e^{\tau_j}-1}{e^{\tau_j} C_i}N+1\right)/ \tau_j \rceil$.
	Denote $N'= \sum_{k=1}^R t_k \geq N$ the number of iterations of a scheme $\mathcal{S}_{i,j}$. We necessarily have $N'\leq 2e^{1/2}N$ for our choice of $C_i$ and $\tau_j$. Hence the cost of running all methods is of the order of $N(\log_2 N )^2$.
	
	If $\tau = 0$ and $N\geq 2 C^*_{\kappa,0}$, then $i = \lceil \log_2 C^*_{\kappa,0} \rceil \leq \lfloor \log_2 N \rfloor$. Therefore $\mathcal{S}_{i,0}$ has been run and bound \eqref{eq:conv_s_2_gen_schedule_tau=0}  shows then that the last iterate $\hat{x}$ satisfies
	\[
	\vspace{-1ex}
	f(\hat{x}) -f^* \leq  \left(\frac{c\kappa}{C_i^2}\right)^\frac{N}{C_i} (f(x_0)-f^*).
	\]
	Using that $C^*_{\kappa,0} \leq C_i \leq 2C^*_{\kappa,0}$, 
	\BEAS
	f(\hat{x}) -f^* \leq  \left(\frac{c\kappa}{(C^*_{\kappa,0})^2}\right)^\frac{N}{2C^*_{\kappa,0}} (f(x_0)-f^*) \leq  \exp\left(-e^{-1}(c\kappa)^{-\frac{1}{2}} N\right)(f(x_0)-f^*).
	\EEAS
	
	If $\tau \geq \frac{1}{N}$ and $N \geq 2 C^*_{\kappa,\tau}$,  then $j = \lceil -\log_2 \tau \rceil \leq \lceil \log_2 N \rceil$ and $i = \lceil \log_2 C^*_{\kappa,\tau} \rceil \leq \lfloor \log_2 N \rfloor$. Therefore scheme $\mathcal{S}_{i,j}$ has been run. As $C_i \geq C^*_{\kappa,\tau} \geq C(\tau_j) $, where $C(\tau_j)$ is defined in~\eqref{eq:C_alpha}, bound~\eqref{eq:conv_s_2_gen_schedule_tau>0} shows that the last iterate $\hat{x}$ of scheme $\mathcal{S}_{i,j}$ satisfies
	\[
	f(\hat{x})- f^* \leq \frac{f(x_0)-f^*}{\left(\tau_j e^{-\tau_j} C_i^{-1}N+1\right)^\frac{2}{\tau}}.
	\]
	Finally, by definition of $i$ and $j$, $2\tau_j \geq \tau$ and $C_i \leq 2C^*_{\kappa,\tau}$, so 
	\BEAS
	f(\hat{x})- f^* \leq \frac{f(x_0)-f^*}{\left(\tau e^{-\tau_j} (C^*_{\kappa,\tau})^{-1}N/4 +1\right)^\frac{2}{\tau}}  =  \frac{f(x_0)-f^*}{\left(\tau e^{-1}(c\kappa)^{-\frac{1}{2}} (f(x_0)-f^*)^{\frac{\tau}{2}}  N/4 +1\right)^\frac{2}{\tau}},
	\EEAS
	where we concluded by expanding $C^*_{\kappa,\tau} = e^{1-\tau}(c\kappa)^\frac{1}{2}(f(x_0)-f^*)^{-\frac{\tau}{2}}$ and using that $\tau \geq \tau_j$.
	
	If $\frac{1}{N} > \tau >0$ and $N >2C^*_{\kappa,\tau}$, then $i = \lceil \log_2 C^*_{\kappa,\tau}\rceil \leq \lfloor \log_2 N \rfloor $, so scheme $\mathcal{S}_{i,0}$ has been run. As in~\eqref{eq:shrp-complx}, its iterates $x_k$ satisfy, with $1-\tau=2/r$,
	\BEAS
	f(x_k)-f^* & \leq & \frac{c \kappa}{C_i^2} (f(x_{k-1})-f^*)^\frac{2}{r} \\
	& \leq & \left(\frac{c \kappa}{C_i^2}\right)^{\left(1-(1-\tau)^k\right)/\tau} (f(x_0)-f^*)^{(1-\tau)^k} \\
	& \leq & \left(\frac{c \kappa(f(x_0)-f^*)^{-\tau}}{C_i^2}\right)^{\left(1-(1-\tau)^k\right)/\tau} (f(x_0)-f^*).
	\EEAS
	Now $C_i \geq C^*_{\kappa,\tau} = e^{1-\tau}(c\kappa)^\frac{1}{2}(f(x_0)-f^*)^{-\frac{\tau}{2}}$ and $C_i R \geq N$, therefore last iterate $\hat{x}$ satisfies
	\BEAS
	f(\hat{x})-f^* & \leq & \exp\left(-2(1-\tau) \frac{1-(1-\tau)^{N/C_i}}{\tau} \right) (f(x_0)-f^*).\\
	\EEAS
	As $N\geq C_i$, since
	$
	h(\tau) = \frac{(1-\tau)\left(1-(1-\tau)^\frac{N}{C_i}\right)}{1-(1-\tau)}
	$
	is decreasing with $\tau$ and $\frac{1}{N} > \tau >0$, we have 
	\BEAS
	f(\hat{x})-f^* & \leq & \exp\left(- 2(N-1)\left(1-\left(1-\frac{1}{N}\right)^{N/C_i}\right) \right) (f(x_0)-f^*) \\
	& \leq & \exp\left(- 2(N-1)\left(1-\exp\left(-\frac{1}{C_i}\right)\right) \right) (f(x_0)-f^*) \\
	& \leq & \exp\left(- 2\frac{N-1}{C_i}\left(1-\frac{1}{2C_i}\right) \right) (f(x_0)-f^*).
	\EEAS
	having used the facts that $ (1+a x)^{\frac{b}{x}} \leq \exp(a b)$ if $ax\geq -1$, $\frac{b}{x}\geq 0$ and $1-x+\frac{x^2}{2} \geq \exp(-x)$ when $x\geq 0$. As $C_i = 2^i \geq 1$, we finally get
	\BEAS
	f(\hat{x})-f^* & \leq & \exp\left(-\frac{N-1}{C_i}\right) (f(x_0)-f^*) \\
	& \leq & \exp\left(-\frac{N-1}{2C^*_{\kappa,\tau}}\right) (f(x_0)-f^*) \\
	& \leq & \frac{f(x_0)-f^*}{\left( \tau (C^*_{\kappa,\tau})^{-1} (N-1) /4+1\right)^\frac{2}{\tau}} \\
	& \leq & \frac{f(x_0)-f^*}{\left(\tau (f(x_0)-f^*)^{\frac{\tau}{2}} e^{-1}(c\kappa)^{-\frac{1}{2}} (N-1)/4 +1\right)^\frac{2}{\tau}}.
	\EEAS
	using the fact that $e^\tau \geq1$.
\end{proof}
 
In the strongly convex case, this adaptive bound is similar to the one of \citep{Nest13b} to optimize smooth strongly convex functions in the sense that we lose approximately a log factor of the condition number of the function. However our assumptions are weaker and our bound also handles all sharpness regimes, i.e. any exponent $r\in[2,+\infty]$, not just the strongly convex case. Finally the step size chosen for the grid search was set to 2. The proof can be adapted for a generic step size $h$, the size of the grid may be reduced but corresponding bounds will suffer an $h^2$ approximation loss compared to the best schedule. 

Note that the scheduled restart schemes we present here adapt to a global sharpness hypothesis on the sublevel set defined by the initial point and are not locally adaptive to potentially better constant~$\mu$ on smaller sublevel sets. On the other hand, restart schemes based on a primal gap, presented in Section~\ref{sec:adaptive}, do adapt to the local value of $\mu$, although these schemes require having access to the primal gap.

\subsection{Comparison to gradient descent}\label{sec:compagrad}
We end this section by analyzing the behavior of gradient descent in light of the sharpness assumption in order to compare the advantage of restarted accelerated method to plain gradient descent. While the bounds we obtain using the basic gradient method are suboptimal compared to the ones above, the gradient algorithm having no memory will automatically adapt to the best ``restart'' schedule.
Given only the smoothness hypothesis, the gradient descent algorithm, presented in e.g. \citep{Nest15}, starts from a point $x_0$ and outputs iterates 
\[
x_t = \mathcal{G}(x_0,t)\quad \mbox{such that} \quad f(x_t) -f^* \leq \frac{L}{t}d(x_0,X^*)^2.
\]
While accelerated methods use the last two iterates to compute the next one, simple gradient descent algorithms use only the last iterate, so the algorithm can be seen as (implicitly) restarting at each iteration. Formally we use that its convergence can be bounded as, for $k\geq 1$,
\BEQ\label{eq:grad_bound}
f(x_{k+t}) -f^* \leq \frac{L}{t}d(x_k,X^*)^2.
\EEQ
and we analyze it in light of the restart interpretation using the sharpness property.
\begin{proposition}
	Let $f$ be a convex function and $x_0 \in \dom f$. Denote $K = \{x: f(x) \leq f(x_0) \}$ and assume that $f \in \mathcal{H}_{2, L} \cap \mathcal{L}_{r, \mu}(K)$.
	Denote by $x_t = \mathcal{G}(x_0,t)$ the iterate sequence generated by the gradient descent algorithm started at $x_0$ to solve \eqref{eq:optim_pb_intro} and define 
	\[
	t_k = e^{1-\tau} c \kappa (f(x_0)-f^*)^\tau e^{\tau k},
	\]
	with $\kappa$ and $\tau$ defined in~\eqref{eq:kappa_tau} and $c = e^{2/e}$ here.
	The precision reached after $N=\sum_{k=1}^n t_k$ iterations is given by, 
	\begin{align*}
	f(x_N) - f^* \leq \exp\left(- e^{-1} (c\kappa)^{-1}N\right) (f(x_0)-f^*) = O\left(\exp(-\kappa^{-1} N)\right) , \quad \mbox{when $\tau =0$,}
	\end{align*}
	while, 
	\begin{align*}
	f(x_N) - f^* \leq \frac{f(x_0) -f^*}{\left(\tau e^{-1} (c\kappa)^{-1} (f(x_0)-f^*)^{\tau}  N+1\right)^\frac{1}{\tau}} =  O\left(N^{-\frac{1}{\tau}} \right) , \quad \mbox{when $\tau >0$.}
	\end{align*}
\end{proposition}
\begin{proof}
	For a given $\gamma \geq 0$, we construct a subsequence $x_{\phi(k)}$ of $x_t$ such that 
	\BEQ\label{eq:grad_lin_conv}
	f(x_{\phi(k)}) -f^* \leq e^{-\gamma k} (f(x_0)-f^*).
	\EEQ
	Define $x_{\phi(0)} = x_0$. Assume that \eqref{eq:grad_lin_conv} is true at iteration $k-1$, then combining complexity bound~\eqref{eq:grad_bound} and~\eqref{eq:KL}, for any $t\geq 1$,
	\BEAS
	f(x_{\phi(k-1)+t}) -f^* & \leq & \frac{c\kappa}{t}(f(x_{\phi(k-1)})-f^*)^\frac{2}{r} \\
	& \leq & \frac{c\kappa}{t}e^{-\gamma\frac{2}{r}(k-1)}(f(x_0)-f^*)^\frac{2}{r}.
	\EEAS
	where $c = e^{2/e}$, using that $r^{2/r}\leq e^{2/e}$.
	Taking $t_k = e^{\gamma(1-\tau)} c \kappa (f(x_0)-f^*)^{-\tau}e^{\gamma \tau k}$ and $\phi(k) = \phi(k-1)+t_k$  ensures that~\eqref{eq:grad_lin_conv} holds at iteration $k$. Using Lemma~\ref{lem:total_number}, we obtain at iteration $N= \phi(n) = \sum_{k=1}^n t_k$,
	\[
	f(x_N) - f^* \leq \exp\left(-\gamma e^{-\gamma} (c\kappa)^{-1} N\right) (f(x_0)-f^*),\quad\mbox{if $\tau = 0$,}
	\]
	and 
	\[
	f(x_N) - f^* \leq \frac{f(x_0) -f^*}{\left(\tau \gamma e^{-\gamma} (c\kappa)^{-1} (f(x_0)-f^*)^{\tau} N+1\right)^\frac{1}{\tau}},\quad \mbox{if $\tau>0$}.
	\]
	These bounds are minimal for $\gamma=1$ and the results follow.
\end{proof}

We observe that restarting accelerated gradient methods reduces complexity from $O(\epsilon^{-\tau})$ to $O(\epsilon^{-\tau/2})$ compared to simple gradient descent. More general results on the convergence of (sub)gradient descent algorithms under a {\L}ojasiewicz inequality assumption were developed by \citet{Bolt15}. 
	
	\section{Universal Scheduled Restarts for Convex Problems}\label{sec:scheduled_gen}
	In this section, we generalize previous results to $s$-smooth functions as defined in Definition~\ref{def:holder} to tackle both smooth and non-smooth convex optimization problems.
Without further assumptions on $f$, the optimal rate of convergence for this class of functions is bounded as $O(1/N^{\rho})$, where $N$ is the total number of iterations and
\BEQ\label{def:q}
\rho = 3s/2-1,
\EEQ 
which gives $\rho =2$ for smooth functions and $\rho = 1/2$ for non-smooth functions. 
The universal fast gradient method \citep{Nest15} achieves this rate by requiring only a target accuracy $\epsilon$ and a starting point $x_0$. It outputs after $t$ iterations a point
\BEQ\label{eq:algobound_gen}
 x = \mathcal{U}(x_0,\epsilon,t), \quad\mbox{such that} \quad f(x) - f^* \leq  \frac{\epsilon}{2} + \frac{cL^\frac{2}{s} d(x_0,X^* )^2}{\epsilon^{\frac{2}{s}}t^{\frac{2\rho}{s}}}  \frac{\epsilon}{2},
\EEQ
where $c$ is a constant ($c = 2^\frac{4s-2}{s}$). A simplified implementation of the universal fast gradient method that enforces monotonicity in objective values of the outputs of the algorithm is presented in Appendix~\ref{sec:algos}.

We assume again that $f$ satisfies a {\L}ojasiewicz growth condition on its initial sublevel set.
The key difference with the smooth case described in the previous section is that here we schedule {\em both} the target accuracy $\epsilon_k$ used by the algorithm {\em and} the number of iterations $t_k$ made at the $k\textsuperscript{th}$ run of the algorithm. Our scheme is described in Algorithm~\ref{algo:scheduled_gen}.

\begin{algorithm}[h]
	\caption{Universal scheduled restarts for convex minimization \label{algo:scheduled_gen}}
	\begin{algorithmic}		
		\State{\textbf{Inputs :} $x_0\in\reals^n$, $\epsilon_0 \geq f(x_0)-f^*$, $\gamma \geq 0$ and a sequence $t_k$ for $k=1,\ldots,R$.}
		\For{$k=1,\ldots,R $}
		\vspace*{-0.3cm}
		\State{\[ 
			\epsilon_k  :=  e^{-\gamma} \epsilon_{k-1}, \qquad
			x_k  :=  \mathcal{U}(x_{k-1},\epsilon_k,t_k) 
			\]}
		\vspace*{-0.5cm}
		\EndFor
		\State{\textbf{Output :} $\hat{x} := x_R$}
	\end{algorithmic}
\end{algorithm}
Our strategy is to choose a sequence $t_k$ that ensures
\[
f(x_k) - f^* \leq \epsilon_k,
\]
for the geometrically decreasing sequence $\epsilon_k$. The overall complexity of our method will then depend on the growth of $t_k$ as described in Lemma~\ref{lem:total_number}.
\begin{proposition}\label{th:sched_gen}
	Let $f$ be a convex function and $x_0 \in \dom f$. Denote $K = \{x: f(x) \leq f(x_0) \}$ and assume that $f \in \mathcal{H}_{s, L} \cap \mathcal{L}_{r, \mu}(K)$.
	Run Algorithm~\ref{algo:scheduled_gen} from $x_0$ for  a given $\epsilon_0 \geq f(x_0)-f^*$ with  
	\[ 
	\gamma =\rho, \qquad t_k = C^*_{\kappa,\tau,\rho}e^{\tau k},\quad \mbox{where}\quad  C^*_{\kappa,\tau,\rho} \triangleq e^{1-\tau}(c\kappa)^\frac{s}{2\rho}\epsilon_0^{-\frac{\tau}{\rho}} 
	\]
	where $\rho$ is defined in \eqref{def:q}, $\kappa$ and $\tau$ are defined in~\eqref{eq:kappa_tau} and $c = 8e^{2/e}$ here. The precision reached at the last point $\hat{x}$ is given by, 
	\begin{align*}
	f(\hat{x}) - f^* ~\leq~ \exp\left(-\rho e^{-1} (c\kappa)^{-\frac{s}{2\rho}} N\right)\epsilon_0 ~=~ O\left(\exp(-\kappa^{-\frac{s}{2\rho}} N)\right), \quad \mbox{when $\tau =0$,}
	\end{align*}
	while, 
	\begin{align*}
	f(\hat{x})-f^* ~\leq~ \frac{\epsilon_0}
	{
		\left(\tau e^{-1} (c\kappa)^{-\frac{s}{2\rho}}\epsilon_0^{\frac{\tau}{\rho}}N
		+1
		\right)^{-\frac{\rho}{\tau}}} ~=~ O\left(N^{-\frac{\rho}{\tau}}\right), \quad \mbox{when $\tau >0$,}
	\end{align*}
	where $N=\sum_{k=1}^R t_k$ is total number of iterations.
\end{proposition}
\begin{proof} 
	Our goal is to ensure that the target accuracy is reached at each restart, i.e. 
	\BEQ
	f(x_k)-f^* \leq \epsilon_k \label{eq:conv_lin}.
	\EEQ
	By assumption, \eqref{eq:conv_lin} holds for $k=0$. Assume that \eqref{eq:conv_lin} is true at iteration $k-1$, combining~\eqref{eq:KL} with the complexity bound in~\eqref{eq:algobound_gen}, then 
	\begin{align}\label{eq:shrp-complx_gen}
	f(x_k) - f^* & \leq   \frac{\epsilon_k}{2} + \frac{c\kappa(f(x_{k-1}) -f^*)^\frac{2}{r}}{\epsilon_k^{\frac{2}{s}}t_k^{\frac{2\rho}{s}}} \frac{\epsilon_k}{2} \leq  \frac{\epsilon_k}{2} + \frac{c\kappa}{t_k^{\frac{2\rho}{s}}} \frac{\epsilon_{k-1}^\frac{2}{r}}{\epsilon_k^\frac{2}{s}} \frac{\epsilon_k}{2},
	\end{align}
	where $c = 8e^{2/e}$ using that $r^{2/r} \leq e^{2/e}$.
	By definition $\epsilon_k= e^{-\gamma k} \epsilon_0$, so to ensure \eqref{eq:conv_lin} at iteration $k$ this imposes
	\[
	\frac{c\kappa e^{\gamma\frac{2}{r}} e^{-\gamma \left(\frac{2}{r}-\frac{2}{s}\right)k} }{ t_k^{\frac{2\rho}{s}} }  \epsilon_0^{\frac{2}{r}-\frac{2}{s}} \leq 1.
	\]
	Rearranging terms in last inequality, using $\tau$ defined in \eqref{eq:kappa_tau}, 
	\[
	t_k \geq e^{\gamma\frac{1-\tau}{\rho}}(c \kappa)^{\frac{s}{2\rho}} \epsilon_0^{-\frac{\tau}{\rho}}  e^{ \frac{\gamma \tau}{\rho} k }.
	\]
	Choosing $t_k = C e^{\alpha k}$, where
	\[
	C =  e^{\gamma\frac{1-\tau}{\rho}}(c \kappa)^{\frac{s}{2\rho}} \epsilon_0^{-\frac{\tau}{\rho}} \qquad \mbox{and} \qquad
	\alpha  =  \frac{\gamma \tau}{\rho},
	\]
	and using Lemma~\ref{lem:total_number} then yields, 
	\BEQ\label{eq:sched_gen_conv_tau=0}
	f(\hat{x}) -f^* \leq \exp(- \gamma e^{-\frac{\gamma}{\rho}} (c \kappa)^{-\frac{s}{2\rho}} N) \epsilon_0,
	\EEQ
	when $\tau =0$, while, 
	\BEQ\label{eq:sched_gen_conv_tau>0}
	f(\hat{x}) -f^* \leq  \quad \frac{\epsilon_0}
	{
		\left(\frac{\gamma \tau}{\rho} e^{-\frac{\gamma}{\rho}} (c\kappa)^{-\frac{s}{2\rho}}\epsilon_0^{\frac{\tau}{\rho}}N
		+1
		\right)^\frac{\rho}{\tau}}.
	\EEQ
	when $\tau >0$. These bounds are minimal for $\gamma=\rho$ and the results follow.
\end{proof}
This bound matches the lower bounds for optimizing smooth and sharp functions up to constant factors~\citep[Eq. 1.21]{Nemi85}. Notice that, compared to \citet{Nemi85}, we can tackle non-smooth convex optimization by using the universal fast gradient algorithm of \citet{Nest15}. 
The rate of convergence in Proposition~\ref{th:sched_gen} is controlled by the ratio between $\tau$ and $\rho$. If these are unknown, a log-scale grid search will not be able to reach the optimal rate, even if $\rho$ is known since we will miss the optimal rate by a constant factor, see Appendix~\ref{sec:universal_grid_search}. If both are known, in the case of non-smooth strongly convex functions for example, a grid-search on $C$ recovers a nearly optimal bound. Finally note that our bound is provided with respect to the number of iterations of the accelerated algorithms, the corresponding bounds in terms of numbers of calls to the oracles can be found by analyzing the line-search cost of the fast universal gradient method.
	
	\section{Restart With Known Primal Gap}\label{sec:adaptive}
	Here, we assume that we know the optimum $f^*$ of \eqref{eq:optim_pb_intro}.
This is the case for example in zero-sum matrix game problems or over-parametrized least-squares without regularization. We assume again that $f$ satisfies the generic smoothness assumption \eqref{eq:smooth} and the {\L}ojasiewicz growth condition \eqref{eq:KL} on its initial sublevel set. 
We use again the universal gradient method $\mathcal{U}$. 
Here however, we can stop the algorithm when it reaches the target accuracy as we know the optimum $f^*$,
i.e. we stop after $t_\epsilon$ inner iterations such that $x = \mathcal{U}(x_0,\epsilon,t_\epsilon)$ satisfies $f(x)-f^* \leq \epsilon$, and write $x \triangleq \mathcal{C}(x_0,\epsilon)$ the output of this method. 

Here we simply restart this method and decrease the target accuracy by a constant factor after each restart. Our scheme is described in Algorithm~\ref{algo:crit}. The following proposition describes its convergence. 
\begin{algorithm}[H]
	\caption{Restart with known primal gap for convex minimization \label{algo:crit}}
	\begin{algorithmic}
		\State{\textbf{Inputs :} $x_0 \in \mathbb{R}^n, f^*, \gamma \geq 0, \epsilon_0 = f(x_0)-f^*$}
		\For{$k=1,\ldots,R$}
		\vspace*{-0.3cm}
		\State{\[
			\epsilon_k  :=  e^{-\gamma} \epsilon_{k-1}, \qquad
			x_k  :=  \mathcal{C}(x_{k-1},\epsilon_k) 
			\]}
		\vspace*{-0.5cm}
		\EndFor
		\State{\textbf{Output :} $\hat{x} := x_R$}
	\end{algorithmic}
\end{algorithm}
\begin{proposition}\label{th:crit}
	Let $f$ be a convex function and $x_0 \in \dom f$. Denote $K = \{x: f(x) \leq f(x_0) \}$ and assume that $f \in \mathcal{H}_{s, L} \cap \mathcal{L}_{r, \mu}(K)$. 
	Run Algorithm~\ref{algo:crit} from $x_0$ with parameter $\gamma = 1$. The precision reached at the last point $\hat x$ is given by,
	\begin{align*}
	f(\hat{x}) - f^* ~\leq~  \exp\left(- e^{-\frac{1}{\rho}} (c\kappa)^{-\frac{s}{2\rho}} N\right)(f(x_0)-f^*) ~=~  O\left(\exp(-\kappa^{-\frac{s}{2\rho}} N)\right) , \quad \mbox{when $\tau =0$,}
	\end{align*}
	while,
	\[
	f(\hat{x})-f^* ~\leq~  \frac{f(x_0)-f^*}
	{
		\left(\frac{\tau}{\rho} e^{-\frac{1}{\rho}} (c\kappa)^{-\frac{s}{2\rho}}(f(x_0)-f^*)^{\frac{\tau}{\rho}}N
		+1
		\right)^\frac{\rho}{\tau}} ~=~   O\left(N^{-\frac{\rho}{\tau}}\right), \quad \mbox{when $\tau >0$,}
	\]
	where $N$ is the total number of iterations, $\rho$ is defined in \eqref{def:q}, $\kappa$ and $\tau$ are defined in~\eqref{eq:kappa_tau} and $c = 8e^{2/e}$ here. Those bounds are suboptimal to the best scheduled restarts by a factor at most $e/2 \approx 1.3$.
\end{proposition}
\begin{proof}
	Given $\gamma \geq 0$, the linear convergence of our scheme is ensured by our choice of target accuracies $\epsilon_k$. It remains to compute the number of iterations $t_{\epsilon_k}$  needed by the algorithm before the $k\textsuperscript{th}$ restart. Following the proof of Proposition~\ref{th:sched_gen}, for $k\geq 1$ we know that the target accuracy is necessarily reached after
	\[
	\bar{t}_k = e^{\gamma\frac{1-\tau}{\rho}}(c \kappa)^{\frac{s}{2\rho}} \epsilon_0^{-\frac{\tau}{\rho}}  e^{ \frac{\gamma \tau}{\rho} k }
	\]
	iterations, such that $t_{\epsilon_k} \leq \bar{t}_k$. So Algorithm~\ref{algo:crit} achieves linear convergence while needing less inner iterates than the scheduled restart presented in Proposition~\ref{th:sched_gen}, its convergence is therefore at least as good. For a given $\gamma$ bounds~\eqref{eq:sched_gen_conv_tau=0} and~\eqref{eq:sched_gen_conv_tau>0} follow with $\epsilon_0 = f(x_0)-f^*$. The dependency in $\gamma$ of the restart scheme in bounds~\eqref{eq:sched_gen_conv_tau=0} and~\eqref{eq:sched_gen_conv_tau>0} is a factor 
	\[
	h(\gamma) = \gamma e^{-{\gamma}/{\rho}}
	\]
	of the number of iterations, whose maximum value is reached for $\gamma=\rho$. Taking $\gamma = 1$, then leads to a bound suboptimal by a constant factor of at most $h(\rho)/h(1) \leq e/2 \approx 1.3$ for $\rho \in[1/2,2]$, so running this scheme with $\gamma =1$ makes it parameter-free while producing nearly optimal complexity.
\end{proof}

When $f^*$ is known, the above restart scheme is adaptive, contrary to the general non-smooth case in Proposition~\ref{th:sched_gen}. It can even adapt to the local values of $L$ or $\mu$ as we use a criterion instead of a preset schedule. Here, stopping using $f(x_k)-f^*$ implicitly yields optimal choices of $C$ and $\tau$. Note that this approach generalizes to algorithms for which a bound on the primal gap is available as in the Frank-Wolfe algorithm, see~\citep{kerdreux2019restarting}. 
	
	\section{Extensions}\label{sec:compobreg}
	Previous analyses of restart schemes only require bounds of the form~\eqref{eq:algobound_s=2} or ~\eqref{eq:algobound_gen}. Our results extend then readily to non-Euclidean composite settings or structured objectives as presented below. 

\subsection{Composite Problems \& Bregman Divergences}
We extend previous schemes to more general convex optimization problems of the form
\BEQ\label{eq:pb_composite}
\mbox{minimize}~ f(x) \triangleq \phi(x) + g(x),
\EEQ
where $g$ is a simple convex function (the meaning of simple will be clarified later), $\phi$ is a convex $s$-smooth function with respect to a given norm $\|\cdot\|$ (potentially non-Euclidean) as defined below,  and $\phi$ is defined on an open-set containing $\dom g$, i.e. $\dom f = \dom g$.

\begin{definition}
	A function $\phi$ is $s$-smooth for a given $1\leq s\leq 2$ with respect to a norm $\|\cdot\|$ if there exists a constant $L>0$ such that
	\BEQ\label{eq:gen_smooth}\nonumber
	\|\nabla \phi(x)-\nabla \phi(y) \|_* \leq L \|x-y\|^{s-1}, \quad  \mbox{for all $x,y \in \dom \phi$}
	\EEQ
	and any subgradients $\nabla \phi(x) \in \partial \phi(x), \nabla \phi(y) \in \partial \phi(y)$ of $\phi$ at $x, y$ respectively, with $\|\cdot\|_*$ being the dual norm of $\|\cdot\|$.
	We denote by $\mathcal{H}_{s, L, \|\cdot\|}$ the set of $s$-smooth functions with respect to a norm $\|\cdot\|$ with parameter $L$.
\end{definition}
To exploit the smoothness of $\phi$ with respect to a generic norm, we assume that we have access to a potential function $h$ with $\dom(f) \subset\dom(h)$, strongly convex with respect to the norm $\|\cdot\|$ with convexity parameter equal to one, which means
\[ 
h(y) \geq h(x) + \nabla h(x)^T(y-x) + \frac{1}{2} \|x-y\|^2, \quad \mbox{for any $x,y \in \dom(h)$}. 
\]
We define the Bregman divergence associated to $h$ as, for given $x,y \in \dom(h)$,
\[ 
D_h(y;x) = h(y) - h(x) -\nabla h(x)^T(y-x) \geq  \frac{1}{2} \|x-y\|^2.
\]
For $h(x) = \frac{1}{2}\|x\|_2^2$, we get $D_h(y;x) = \frac{1}{2}\|x-y\|_2^2$ and recover the Euclidean setting. Given the problem geometry, appropriate choices of potential functions and associated Bregman divergences can lead to significant performance gains in high dimensional settings. 
We now formally state the assumption that $g$ is simple. Given $x,y \in \dom(f)$ and $\lambda\geq 0$ we assume that 
\BEQ\label{eq:prox_g}
\min_z \left\{y^Tz + g(z) + \lambda D_h(z;x) \right\}
\EEQ
can be solved either in a closed form or by some fast computational procedure.

This setting includes some constrained optimization problems, where $g$ is the indicator function of a closed convex set, on which we can easily project the points. It also includes sparse optimization problems, such as the LASSO, where $\phi(x)= \|Ax-b\|_2^2$, with $A\in \reals^{m\times n}$, $ b \in \reals^m$,  $g(x) = \lambda \|x\|_1$, with $\lambda \geq 0$ and $h(x) = \frac{1}{2}\|x\|_2^2$.
To apply our analysis of restart schemes we need two things: an accelerated algorithm and an appropriate notion of sharpness. In the spirit of \citep{Baus16,Lu16}, we thus introduce the notion of relative sharpness.
\begin{definition}
	A function $f$ satisfies a relative {\L}ojasiewicz growth condition  with respect to a strictly convex function $h$ on a set $K$ if there exist $r\geq 1$, $\mu >0$ such that 
	\BEQ\label{eq:relative_sharp}
	\frac{\mu}{r} D_h(x,X^*)^\frac{r}{2} \leq f(x) - f^* \quad  \mbox{for any $x \in K$} 
	\EEQ
	where $D_h(x,X^*) = \min_{x^* \in X^*} D_h(x^*; x)$ and $D_h$ is the Bregman divergence associated to $h$. We denote by $\mathcal{L}_{r, \mu, h}(K)$ the set of functions satisfying a relative  {\L}ojasiewicz growth condition w.r.t to $h$ on a set $K$ with parameters $r, \mu$.
\end{definition}
If $h = \frac{1}{2}\|x\|_2^2$ we recover the definition of the {\L}ojasiewicz growth in the Euclidean setting (with slightly modified constants). This assumption is as generic as our first one in~\eqref{eq:KL} as it is satisfied if $f$ and $h$ are subanalytic \citep[Th. 6.4]{Bier88}. 

The algorithms are essentially the same as before, except that the distance to the set of minimizers is replaced by the Bregman divergence to the set of minimizers. We keep the same notations for the algorithms as the implementations are the same as presented in Section~\ref{sec:algos}. Formally, if $\phi$ is smooth with respect to a norm $\|\cdot\|$, the accelerated algorithm outputs after $t$ iterations a point 
\BEQ\label{eq:acc_breg_bound}
x = \mathcal{A}(x_0,t) \qquad \mbox{such that} \qquad f(x) - f^* \leq  \frac{cL}{t^2} D_h(x_0,X^*),
\EEQ
where $c = 8$ here. The next Corollary generalizes Proposition~\ref{th:sched}.
\begin{restatable}{corollary}{bregsmooth}\label{cor:sched}
Let $f = \phi +g$ be a composite convex function, $x_0 \in \dom f$ and $K = \{x : f(x) \leq f(x_0)\}$. Assume that $\phi \in \mathcal{H}_{2, L, \|\cdot\|}$ for a given norm $\|\cdot\|$, that $f \in \mathcal{L}_{r, \mu, h}(K)$ for $h$ strongly convex with respect to $\|\cdot\|$ and that $g$ is simple such that problems~\eqref{eq:prox_g} can be computed efficiently. Run Algorithm~\ref{algo:scheduled_s=2} from $x_0$ with iteration schedule $t_k=C^*_{\kappa,\tau}e^{\tau k}$, for $k=1,\ldots,R$,
where 
\[
C^*_{\kappa,\tau} \triangleq e^{1-\tau}(c\kappa)^\frac{1}{2}(f(x_0)-f^*)^{-\frac{\tau}{2}},
\]
with $\kappa$ and $\tau$ defined in~\eqref{eq:kappa_tau} and $c=8e^{2/e}$. The precision reached at the last point $\hat{x}$ is given by,
\begin{align*}
f(\hat{x}) - f^* \leq \exp\left(-2e^{-1}(c\kappa)^{-\frac{1}{2}} N\right)(f(x_0)-f^*) = O\left(\exp(-\kappa^{-\frac{1}{2}} N)\right),  \quad \mbox{when $\tau =0$,}
\end{align*}
while,
\begin{align*}
f(\hat{x}) -f^* \leq \frac{f(x_0)-f^*}{\left(\tau e^{-1}(f(x_0)-f^*)^{\frac{\tau}{2}}  (c\kappa)^{-\frac{1}{2}}N +1\right)^\frac{2}{\tau}} = O\left(N^{-\frac{2}{\tau}}\right), \quad \mbox{when $\tau >0$,}
\end{align*}
where $N=\sum_{k=1}^R t_k$ is the total number of iterations.
\end{restatable}
\begin{proof}
	The proof of Proposition~\ref{th:sched} only relies on the bound in~\eqref{eq:shrp-complx} that combines the growth condition~\eqref{eq:KL} with the complexity bound in~\eqref{eq:algobound_s=2}. For the case with composite problems and Bregman divergences we combine~\eqref{eq:relative_sharp} with the bound~\eqref{eq:acc_breg_bound}, which ensures for the $k$\textsuperscript{th} iterate of the restart scheme,
	$
	f\left(x_k\right)-f^* \leq  \frac{c \kappa}{t_k^2}(f\left(x_{k-1}\right)-f^*)^\frac{2}{r},
	$
	with here $c=8e^{2/e}$. The rest of the proof follows as in Proposition~\ref{th:sched}.
\end{proof}
For general convex functions, given a target accuracy $\epsilon$ and an initial point $x_0$, the universal fast gradient method  outputs after $t$ iterations a point 
\begin{equation}\label{eq:univ_breg_bound}
x = \mathcal{U}(x_0,\epsilon,t) \quad \mbox{such that} \quad f(x) - f^* \leq  \frac{\epsilon}{2} + \frac{cL^\frac{2}{s}D_h(x_0,X^*)}{\epsilon^{\frac{2}{s}}t^{\frac{2\rho}{s}}} \frac{\epsilon}{2},
\end{equation}
where $c = 2^{\frac{5s-2}{s}}$ here. The following Corollary generalizes then Proposition~\ref{th:sched_gen}.
\begin{restatable}{corollary}{breggen}\label{cor:sched_gen}
	Let $f = \phi +g$ be a composite convex function, $x_0 \in \dom f$ and $K = \{x : f(x) \leq f(x_0)\}$. Assume that $\phi \in \mathcal{H}_{s, L, \|\cdot\|}$ for a given norm $\|\cdot\|$, that $f \in \mathcal{L}_{r, \mu, h}(K)$ for $h$ strongly convex with respect to $\|\cdot\|$ and that $g$ is simple such that problems~\eqref{eq:prox_g} can be computed efficiently. Run Algorithm~\ref{algo:scheduled_gen} from $x_0$ for given $\epsilon_0 \geq f(x_0)-f^*$,  
	\[ 
	\gamma =\rho, \qquad t_k = C^*_{\kappa,\tau,\rho}e^{\tau k},\quad \mbox{where}\quad
	C^*_{\kappa,\tau,\rho} \triangleq e^{1-\tau}(c\kappa)^\frac{s}{2\rho}\epsilon_0^{-\frac{\tau}{\rho}} 
	\]
	where $\rho$ is defined in~\eqref{def:q}, $\kappa$ and $\tau$ are defined in~\eqref{eq:kappa_tau} and $c=16e^{2/e}$. The precision reached at the last point $\hat{x}$ is given by, 
	\begin{align*}
	f(\hat{x}) - f^* ~\leq~ \exp\left(-\rho e^{-1} (c\kappa)^{-\frac{s}{2\rho}} N\right)\epsilon_0 ~=~ O\left(\exp(-\kappa^{-\frac{s}{2\rho}} N)\right), \quad \mbox{when $\tau =0$,}
	\end{align*}
	while,
	\begin{align*}
	f(\hat{x})-f^* ~\leq~  \frac{\epsilon_0}
	{
		\left(\tau e^{-1} (c\kappa)^{-\frac{s}{2\rho}}\epsilon_0^{\frac{\tau}{\rho}}N
		+1
		\right)^\frac{\rho}{\tau}} ~=~ O\left(N^{-\frac{\rho}{\tau}}\right), \quad \mbox{when $\tau >0$,}
	\end{align*}
	where $N=\sum_{k=1}^R t_k$ is total number of iterations.
\end{restatable} 
\begin{proof}
	The proof of Proposition~\ref{th:sched_gen} only relies on the bound in~\eqref{eq:shrp-complx_gen} that combines the growth condition~\eqref{eq:KL} with the complexity bound in~\eqref{eq:algobound_gen}. For the case with composite problems and Bregman divergences we combine~\eqref{eq:relative_sharp} with the bound~\eqref{eq:univ_breg_bound}, which ensures for the $k$\textsuperscript{th} iterate of the restart scheme,
	$
	f(x_k) - f^* \leq   \frac{\epsilon_k}{2} + \frac{c\kappa(f(x_{k-1}) -f^*)^\frac{2}{r}}{\epsilon_k^{\frac{2}{s}}t_k^{\frac{2\rho}{s}}} \frac{\epsilon_k}{2} 
	$
	with here $c=16e^{2/e}$. The rest of the proof follows as in Proposition~\ref{th:sched_gen}.	
\end{proof}
The results regarding adaptive schemes and those for which $f^*$ is known, i.e. Propositions~\ref{prop:grid_search} and~\ref{th:crit} respectively, generalize similarly under the relative growth assumption. Those results apply then to generic $\ell_{1,p}$ regularized prediction problems where $g$ is an $\ell_{1,p}$ norm and $\phi$ is a data-fitting term. Indeed error bounds were proven to hold for those problems by~\citet{Zhou15}. Those error bounds are then equivalent to quadratic growth conditions, i.e.~\eqref{eq:KL} with $r=2$~\citep{Drus18}. Previous works demonstrate then linear convergence of proximal gradient descent~\citep{Bolt15}. Here the restart schemes allow to get accelerated rates similar as for smooth strongly convex problems. Note that adaptive schemes were also developed by~\citep{Ferc17} in that case.

\subsection{Smoothing non-smooth problems}
Our approach extends also to problems that can be smoothed, i.e. problems of the form 
\begin{equation}\label{eq:pb_smoothing}
	\mbox{minimize} \quad f(x) \triangleq \phi(Ax) + g(x)
\end{equation}
where $A \in \reals^{m\times n}$, $g$ is a simple convex function and $\phi$ is a non-smooth convex function whose inf-convolution with some smooth  convex function $\psi$  can be computed analytically, i.e. one has access for any $\mu>0$ to 
\begin{equation}\label{eq:smoothing}
\phi_{\mu \psi^\star}(x) = \sup_{u \in \dom \phi^\star} \left\{u^\top x - \phi^\star(u) - \mu \psi^\star(u)\right\},
\end{equation}
where for a function $f$ we denote by $f^\star$ its convex conjugate. Those problems were notably considered by \citet{Nest05}, who proved that, though they a priori suffer from their non-smoothness, they can still be solved in $O(\varepsilon)$ calls to an oracle by using their structure. Formally, we have access to an algorithm $\mathcal{S}$ that, given an initial point $x_0$ and a target accuracy $\varepsilon$, outputs after $t$ iterations a point 
\begin{equation}\label{eq:smoothing_rate}
x = \mathcal{S}(x_0, \epsilon, t) \quad \mbox{such that} \quad f(x) -f^* \leq  \frac{\epsilon}{2} + \frac{c L_{\psi^\star, A}^2D_h(x, X^*)}{\epsilon^2 t^2} \frac{\epsilon}{2}, \quad \mbox{and} \quad  f(x) \leq f(x_0),
\end{equation}
where $h$ is some potential function and $L_{\psi^\star, A}$ is a smoothing constant, see Appendix~\ref{sec:algos} for more details. The scheduled restarts of this algorithm will follow the same strategy as for the fast universal gradient method as presented in the following proposition.
\begin{restatable}{proposition}{smoothingrestart}\label{prop:smoothing}
	Let $f(x) =\phi(A x) + g(x)$ be a non-smooth objective that can be smoothed using~\eqref{eq:smoothing}, $x_0 \in \dom f$  and $K = \{x : f(x) \leq f(x_0)\}$. Assume that we have access to a smoothing method $\mathcal{S}$ ensuring~\eqref{eq:smoothing_rate} for a given strongly convex function $h$ and that $f \in \mathcal{L}_{r, \mu, h}(K)$.
	Given $\epsilon_0 \geq f(x_0) - f^*$, restart the method $\mathcal{S}$ such that for $k\geq 1$,
	\[
	x_k = \mathcal{S}(x_{k-1}, \epsilon_k, t_k), \qquad \epsilon_k = e^{-1} \epsilon_{k-1}, \qquad t_k = \tilde C^*_{\kappa,\tau}e^{\tau k}, \qquad \tilde C^*_{\kappa,\tau} \triangleq e^{1-\tau}(c\kappa)^\frac{1}{2}\epsilon_0^{-\tau},
	\]
	where $\kappa$ and $\tau$ are defined as in~\eqref{eq:kappa_tau} with $s=1$ and $L_{\psi^*, A}$ in place of $L$.
	
	The precision reached at a point $\hat{x}= x_R$ after $R$ restarts is given by, 
	\begin{align*}
	f(\hat{x}) - f^* ~\leq~ \exp\left(- e^{-1} (c\kappa)^{-\frac{1}{2}} N\right)\epsilon_0 ~=~ O\left(\exp(-\kappa^{-\frac{1}{2}} N)\right), \quad \mbox{when $\tau =0$,}
	\end{align*}
	while,
	\begin{align*}
	f(\hat{x})-f^* ~\leq~  \frac{\epsilon_0}
	{
		\left(\tau e^{-1} (c\kappa)^{-\frac{1}{2}}\epsilon_0^\tau N
		+1
		\right)^\frac{1}{\tau}} ~=~ O\left(N^{-\frac{1}{\tau}}\right), \quad \mbox{when $\tau >0$,}
 	\end{align*}
	where $N=\sum_{k=1}^R t_k$ is total number of iterations.
\end{restatable}
\begin{proof}
	The smoothing method has a bound of the same form as the universal fast gradient method, i.e. we have
	\[
	x = \mathcal{S}(x_0, \epsilon, t) \qquad \mbox{such that} \qquad  f(x) - f^* \leq  \frac{\epsilon}{2} + \frac{cL^\frac{2}{s} D_h(x_0,X^* )^2}{\epsilon^{\frac{2}{s}}t^{\frac{2\rho}{s}}}  \frac{\epsilon}{2},	
	\]
	with here $L= L_{\psi^\star, A}$, $s=1$ and $\rho=1$. The optimal restart schedule and corresponding rates follow then from Proposition~\ref{th:sched_gen} by replacing $s=1$  and $\rho=1$.
\end{proof}
As for the universal fast gradient method, a grid-search will not get optimal rates if $r$ and so $\tau$ is unknown. However if it is known, a grid-search will ensure optimal rates up to a constant factor. It is illustrated for sparse recovery problems by~\citet{Roul15}.

If $f^*$ is known, Proposition~\ref{th:crit} is modified into the following proposition. Note that the resulting restart scheme is the one presented by~\citet{Gilp12} for zero-sum matrix games.
\begin{restatable}{proposition}{smoothingrestartoptknown}\label{prop:smoothing_crit}
	Let $f(x) =\phi(A x) + g(x)$ be a non-smooth objective that can be smoothed using~\eqref{eq:smoothing}, $x_0 \in \dom f$  and $K = \{x : f(x) \leq f(x_0)\}$. Assume that $f^*$ is known, that we have access to a smoothing method $\mathcal{S}$ ensuring~\eqref{eq:smoothing_rate} for a given strongly convex function $h$ and that $f \in \mathcal{L}_{r, \mu, h}(K)$.
Denoting $\epsilon_0 = f(x_0) -f^*$, consider the restart scheme defined by 
\[
x_k = S(x_{k-1}, \epsilon_k, t_k) \quad \mbox{s.t.} \quad  \epsilon_k = e^{-1} \epsilon_0 \quad t_k = \argmin \{t: x= \mathcal{S}(x_{k-1}, \epsilon_k, t) \: \mbox{satisfies} \: f(x) - f^* \leq \epsilon_k\}.
\] 
	The precision reached at a point $\hat{x}= x_R$ after $R$ restarts is given by, 
	\begin{align*}
	f(\hat{x}) - f^* ~\leq~ \exp\left(- e^{-1} (c\kappa)^{-\frac{1}{2}} N\right)\epsilon_0 ~=~ O\left(\exp(-\kappa^{-\frac{1}{2}} N)\right), \quad \mbox{when $\tau =0$,}
	\end{align*}
	while,
	\begin{align*}
	f(\hat{x})-f^* ~\leq~  \frac{\epsilon_0}
	{
		\left(\tau e^{-1} (c\kappa)^{-\frac{1}{2}}\epsilon_0^\tau N
		+1
		\right)^\frac{1}{\tau}} ~=~ O\left(N^{-\frac{1}{\tau}}\right), \quad \mbox{when $\tau >0$,}
	\end{align*}
	where $N=\sum_{k=1}^R t_k$ is total number of iterations and $\kappa$ and $\tau$ are defined as in~\eqref{eq:kappa_tau} with $s=1$ and $L_{\psi^*, A}$ in place of $L$.
\end{restatable}
\begin{proof}
	As in Proposition~\ref{th:crit}, the proposed scheme with a termination criterion on the gap cannot do worse than the optimal scheduled restart. The rate is then given by Proposition~\ref{prop:smoothing}.
\end{proof}
	
	\section{Numerical Results}\label{sec:numres}
	We illustrate our results by testing our adaptive restart schemes, the adaptive scheme \emph{Adap} of  Section~\ref{ss:adapt}, and the scheme with stopping criterion on the primal gap \emph{Crit} in Section~\ref{sec:adaptive}, on several problems to compare them against simple gradient descent (\emph{Grad}), accelerated gradient methods (\emph{Acc}), and the restart heuristic enforcing monotonicity (\emph{Mono}) proposed by \citep{Odon15}. For \emph{Adap} we plot the convergence of the best method found by grid search to compare with the restart heuristic. This implicitly assumes that the grid search is run in parallel with enough servers. For \emph{Crit} we use the optimal $f^*$ found by another solver. This gives an overview of its performance when such information is available. All restart schemes were performed using the accelerated gradient with backtracking line search detailed in Appendix~\ref{sec:algos}, with large dots representing restart iterations. 

In Figure~\ref{fig:classif}, we solve classification problems with various losses on the UCI \emph{Sonar} data set \citep{Asun07}. For the least square loss on sonar data set, we observe much faster convergence of the restart schemes compared to the accelerated method. These results were already observed by \citet{Odon15}. For the logistic loss, we observe that restart does not provide much improvement for a budget of 1000 iterations.
For the hinge loss, we regularized by a squared norm and optimize the dual, which means solving a quadratic problem with box constraints. We observe here that the scheduled restart scheme converges much faster, while restart heuristics may be activated too late. We observe similar results for the LASSO problem.  
This highlights the benefits of a sharpness assumption for these last two problems. In general \emph{Crit} ensures the theoretical accelerated rate but \emph{Adap} exhibits more consistent behavior. Again, precisely quantifying sharpness from  data/problem structure is a key open problem. 

To account for the grid search effort, in Figure~\ref{fig:fair_plot}, we multiplied the number of iterations made by the \emph{Adap} method by the size of the grid. This is for the LASSO problem on \emph{Sonar} data set with a grid step size of $4$. This shows that the benefits of the restart schemes make the grid search effort acceptable both on paper and in practice. More clever grid search strategies for scheduled restarts run in parallel would even reduce this effort.

\begin{figure}[H]
	\begin{center}
		\psfrag{f(x)-f\textsuperscript{*}}[c][c]{\tiny{$f_(x_t)-f^*$}}
		\psfrag{Iterations}[c][c]{\small{Iterations}}
		\includegraphics[width = 0.35\linewidth]{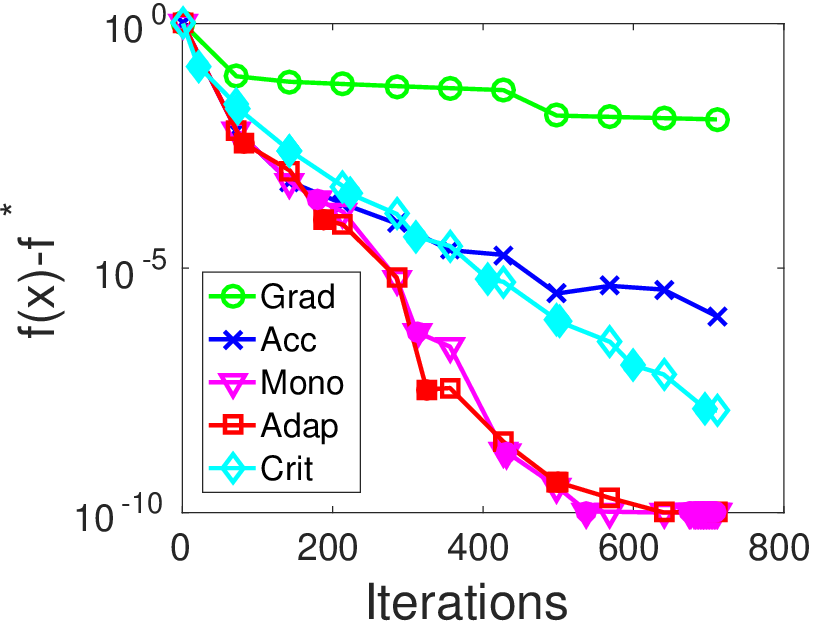}~
		\includegraphics[width = 0.35\linewidth]{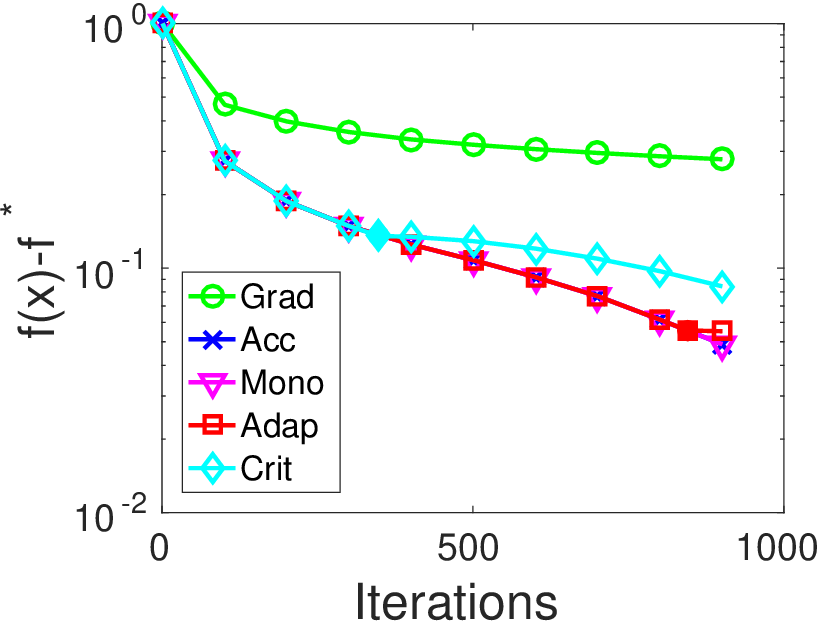}
		
		\includegraphics[width = 0.35\linewidth]{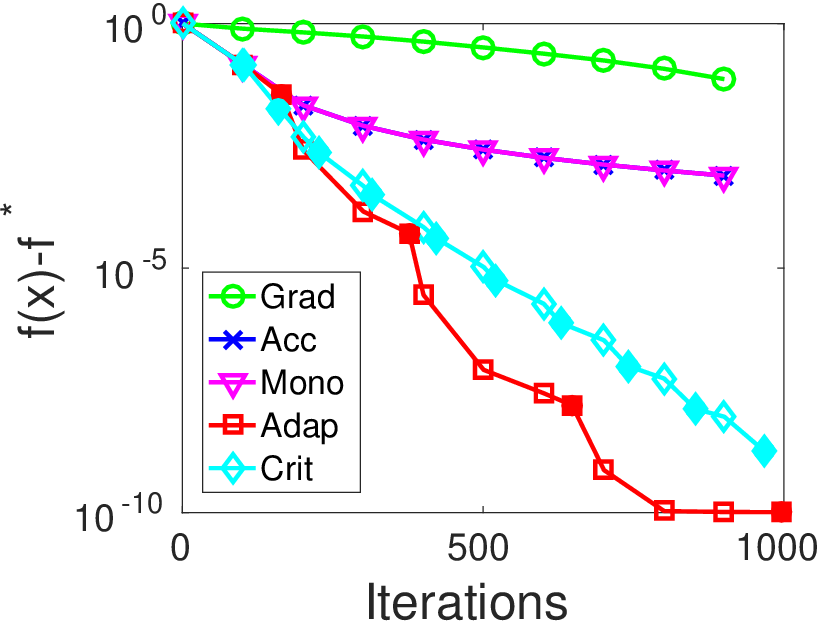}~
		\includegraphics[width = 0.35\linewidth]{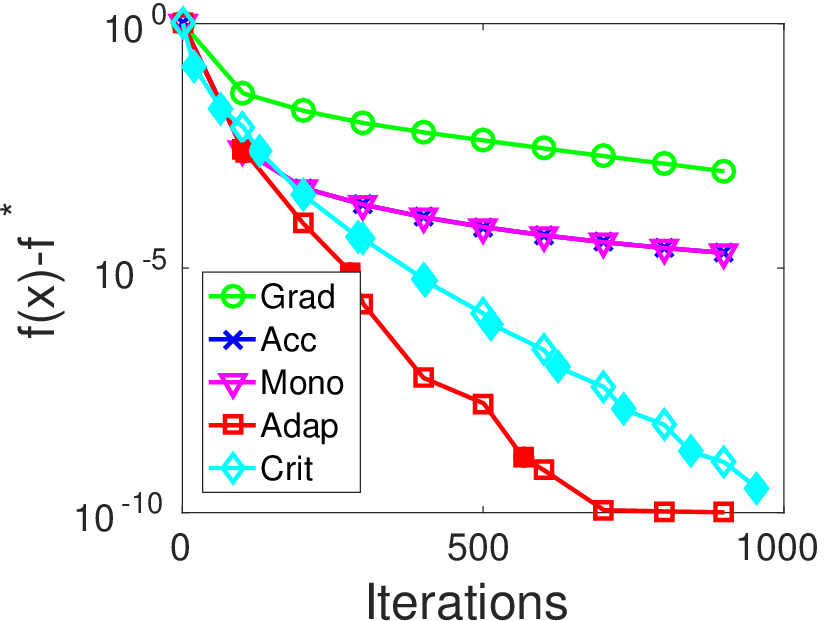}
		\vskip 2ex
		\caption{Sonar data set. From left to right: least square loss, logistic loss, dual SVM problem and LASSO. We use adaptive restarts (Adap), gradient descent (Grad), accelerated gradient (Acc) and restart heuristic enforcing monotonicity (Mono). Large dots represent the restart iterations. Regularization parameters for dual SVM and LASSO were set to one.}
		\label{fig:classif}
	\end{center}
\end{figure}

\begin{figure}[h]
	\label{fig:fair_plot}
	\begin{center}
		\psfrag{f(x)-f^*}[c][c]{\tiny{$f_(x_t)-f^*$}}
		\psfrag{Iterations}[c][c]{\small{Iterations}}
		\includegraphics[width = 0.38\linewidth]{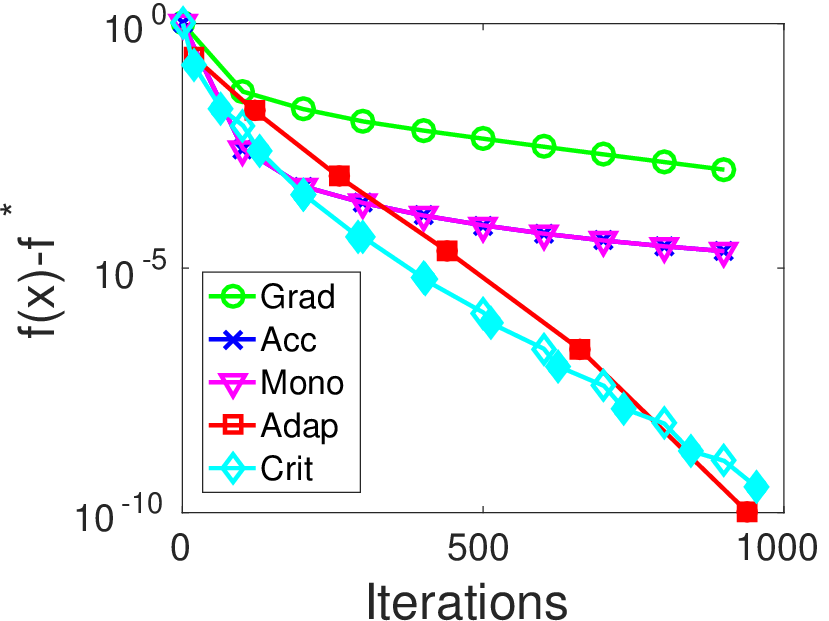}
		\caption{Comparison of the methods for the LASSO problem on Sonar dataset where number of iterations of the Adaptive method is multiplied by the size of the grid. Grid search step size is set to 4.}
	\end{center}
\end{figure}

	\section*{Acknowledgments}
	AA is at CNRS, attached to the D\'epartement d'Informatique at \'Ecole Normale Sup\'erieure in Paris, INRIA-Sierra team, PSL Research University. VR is in the Statistics Department of the University of Washington, working for the Algorithmic Foundations of Data Science Institute. The authors would like to acknowledge support from the {\em fonds AXA pour la recherche}, a gift from Soci\'et\'e G\'en\'erale Cross Asset Quantitative Research, an AMX fellowship and a Google focused award.
	
	\bibliographystyle{agsm}
	\bibliography{bib_restart}

	\appendix	
	\section{Algorithms \& Complexity Bounds} \label{sec:algos}
	We present here the algorithms that we restart. In particular we present a simplified version of the universal fast-gradient method of~\citet{Nest15} and shows how we can enforce monotonicity of the objective values, while keeping the optimal rate. The classical accelerated algorithm for smooth convex function is then derived as a special case of the universal method. 

In both cases the smoothness constant do not need to be known in advance. A line-search is provided whose complexity can be bounded as done in~\citep{Nest15}. In practice when restarting the algorithm we use the last smoothness constant provided by the algorithm.

\subsection{Problem formulation}
We focus on composite optimization problems of the form
\BEQ\label{eq:pb_composite_algo}
\mbox{minimize}\quad  f(x) = \phi(x) + g(x)
\EEQ
where $\phi,g$ are proper lower semi-continuous convex functions on $\reals^ n$ and $\phi$ is defined on an open-set containing $\dom g$, i.e. $\dom f = \dom g$. We denote $\|\cdot\|$ a given norm on $ \reals^ n$ and $\|\cdot\|_*$ the dual norm of $\|\cdot\|$.

We assume that $\phi$ is $s$-smooth with respect to $\|\cdot\|$  for a given $s \in [1, 2]$, i.e.  that there exists $L>0$ such that
\begin{equation}\label{eq:holder_smooth}
	\|\nabla \phi(x) - \nabla \phi(y)\|_* \leq L\|x-y\|^{s-1}
\end{equation}
for all $x,y \in \dom f$ and any $\nabla \phi(x) \in \partial \phi(x), \nabla \phi(y) \in \partial \phi(y)$.
We assume that we have access to a function $h$, differentiable on its domain $\dom h\supseteq \dom f$,  strongly convex with respect to the norm $\|\cdot\|$ with convexity parameter equal to one, i.e. 
\[ h(y) \geq h(x) + \nabla h(x)^T(y-x) + \frac{1}{2} \|x-y\|^2, \quad \mbox{for any $x,y \in \dom h$}. \]
We define the Bregman divergence associated to $h$ as, for given $x,y\in \dom h$, 
\begin{equation}\label{eq:bregman}
D_h(y;x) = h(y) - h(x) -\nabla h(x)^T(y-x) \geq  \frac{1}{2} \|x-y\|^2 .
\end{equation}
Finally we assume that, for any  $x,y\in \dom f $ and $\lambda\geq 0$ we can solve 
\begin{equation}\label{eq:prox} 
\min_z y^Tz + g(z) + \lambda D_h(z;x) 
\end{equation}
either in a closed form or by some cheap computational procedure.
In the following, we denote for any $x,y\in \dom f$, 
\[
\ell_f(x; y) = \phi(y) + \nabla \phi(y)^ \top(x-y) + g(x) 
\]
where $\nabla \phi(y) \in \partial \phi(y)$ is any sub-gradient of $\phi$ at $y$. This partial linearization of the objective is convex and satisfies by convexity of $\phi$, 
\begin{equation}\label{eq:lower_bound}
\ell_f(x; y) \leq f(x) 
\end{equation}

\subsection{Universal fast gradient method}
Our simplified version of the Universal Fast Gradient method is presented in Algorithm~\ref{algo:universal}. Proposition~\ref{prop:universal_conv} shows the convergence of Algorithm~\ref{algo:universal}.  Proposition~\ref{prop:line_search} ensures that the line-searches terminate. The total number of oracle calls can be bounded as done by \citet{Nest15}, using a termination criterion, we only give the complexity in terms of iterations of the algorithm. Monotonicity is enforced by simply taking the best of the new and old iterate at each iteration.

\begin{proposition}\label{prop:universal_conv}
	Consider problem~\eqref{eq:pb_composite_algo} where $\phi$ satisfies~\eqref{eq:holder_smooth} with parameters $(s,L)$. Algorithm~\ref{algo:universal}, started at an initial point $x_0$ for a target accuracy $\epsilon$ and an initial estimate $L_0$, outputs after $t$ iterations a point $x_t$ such that
	\[
	f(x_{t}) - f^ * \leq \frac{\epsilon}{2} +  \frac{2^{\frac{5s-2}{s}} L^{\frac{2}{s}}D_h(x_0, X^*)}{\epsilon^\frac{2}{s}t^{\frac{2\rho}{s}}}\frac{\epsilon}{2} \quad \mbox{and} \quad f(x_t) \leq f(x_0),
	\]
	where $D_h(x,X^*) = \min_{x^* \in X^*} D_h(x^*; x)$ with $X^* = \argmin_x f(x)$.
\end{proposition}
\begin{proof}
	Monotonicity of the objective values is ensured by~\eqref{eq:monotonicity}. We fix in the following $x \in \dom f$. 
	Consider the $k$ \textsuperscript{th} iteration of Algorithm~ \ref{algo:universal} for $k\geq 1$. We have
	\begin{align*}
		f(x_{k}) \stackrel{\eqref{eq:monotonicity}}{\leq} f(\tilde x_{k}) & \stackrel{\eqref{eq:suff_decrease}}{\leq} \ell_f(\tilde x_{k}; y_{k-1}) + \frac{ L_k}{2}\|\tilde x_{k} - y_{k-1}\|^ 2 + \frac{\epsilon\theta_k}{2} \\
		& \stackrel{(i)}{\leq} (1-\theta_k) \ell_f(x_{k-1}; y_{k-1}) + \theta_k \ell_f(z_{k}; y_{k-1}) + \frac{ L_k \theta_k^ 2}{2} \|z_{k}-z_{k-1}\|^ 2 + \frac{\epsilon\theta_k}{2}\\
		& \stackrel{\eqref{eq:bregman}}{\leq} (1-\theta_k) \ell_f(x_{k-1}; y_{k-1}) + \theta_k \left(\ell_f(z_{k}; y_{k-1}) +  L_k \theta_k D_h(z_{k}; z_{k-1})\right) + \frac{\epsilon\theta_k}{2} \\
		& \stackrel{\eqref{eq:three_point}}{\leq}(1-\theta_k) \ell_f(x_{k-1}; y_{k-1}){  + }\theta_k \left(\ell_f(x; y_{k-1}){ +  }L_k \theta_k[D_h(x; z_{k-1}) -D_h(x; z_{k})]\right){ + }\frac{\epsilon\theta_k}{2} \\
		& \stackrel{\eqref{eq:lower_bound}}{\leq} (1-\theta_k)f(x_{k-1}) + \theta_k f(x) +  L_k \theta_k^ 2[D_h(x; z_{k-1}) -D_h(x; z_{k})]+ \frac{\epsilon\theta_k}{2}
	\end{align*}
	where in $(i)$ we used that $\tilde x_{k} = (1-\theta_k)x_{k-1} + \theta_k z_{k}$, the convexity of $\ell_f(\cdot; y_{k-1})$ and $\tilde x_{k} -y_{k-1} = \theta_k(z_{k}-z_{k-1})$. 
	Subtracting $f(x)$ on both sides and dividing by $ L_k\theta_k^2$, we get
	\begin{align}
	\frac{1}{ L_k \theta_k^2}(f(x_{k}) - f(x)) &\leq \frac{1-\theta_k}{ L_k \theta_k^2} (f(x_{k-1}) - f(x)) + D_h(x; z_{k-1}) -D_h(x; z_{k}) + \frac{\epsilon}{2 L_k\theta_k} \nonumber 
	\end{align}
	If $k=1$, we have, using the initialization $\theta_1=1$, $z_0 = x_0$,
	\begin{equation} \label{eq:telescoping_one}
	\frac{1}{ L_1 \theta_1^2}(f(x_{1}) - f(x)) \leq  D_h(x; x_0) -D_h(x; z_{1}) + \frac{\epsilon}{2 L_1\theta_1}
	\end{equation}
	Otherwise we have using the definition of $\theta_k$ in \eqref{eq:update_theta},
	\begin{equation}
\frac{1}{ L_k \theta_k^2}(f(x_{k}) - f(x))  \leq \frac{1}{L_{k-1} \theta_{k-1}^2}(f(x_{k-1}) - f(x)) +  D_h(x; z_{k-1}) -D_h(x; z_{k}) + \frac{\epsilon}{2   L_k \theta_k} \label{eq:telescoping}
	\end{equation}
	Using inequality~\eqref{eq:telescoping} recursively from $k$ to $2$ and inequality~ \eqref{eq:telescoping_one} for $k=1$ we get
	\begin{align*}
		\frac{1}{ L_k \theta_k^2}(f(x_{k}) - f(x)) &\leq D_h(x; x_0) -D_h(x; z_{k+1}) + \sum_{j=1}^{k} \frac{\epsilon}{2 L_j \theta_j} \stackrel{\eqref{eq:sum_theta}}{=} D_h(x; x_0) -D_h(x; z_{k+1}) +\frac{\epsilon}{2 L_k \theta_k^2}
	\end{align*}
	Finally a bound on $L_k \theta_k^2$ can be found by combining Proposition~\ref{prop:line_search} and Proposition~\ref{prop:universal_conv} such that 
	$
	L_k\theta_k^2 \leq \frac{2^{\frac{5s-2}{s}} L^{\frac{2}{s}}}{\epsilon^\frac{2}{s}k^{\frac{3s-2}{s}}}\frac{\epsilon}{2}
	$
	and we get, denoting $\rho = \frac{3s-2}{2}$,
	\begin{equation}\label{eq:bound_univ_x}
	f(x_{k}) - f(x) \leq \frac{\epsilon}{2} +  \frac{2^{\frac{5s-2}{s}} L^{\frac{2}{s}}D_h(x;x_0)}{\epsilon^\frac{2}{s}k^{\frac{2\rho}{s}}}\frac{\epsilon}{2}.
	\end{equation}
	Taking $x\in \argmin_{x\in X^*} D_h(x; x_0)$ concludes the proof.
\end{proof}

\begin{proposition}\label{prop:line_search}
	Consider problem~\eqref{eq:pb_composite_algo} where $\phi$ satisfies~\eqref{eq:holder_smooth} with parameters $(s,L)$. The line-searches of Algorithm~\ref{algo:universal} terminate with 
	\[
	L_k \leq  2\epsilon^\frac{s-2}{s} L^{\frac{2}{s}}\theta_k^\frac{s-2}{s}
	\]
\end{proposition}
\begin{proof}
	Lemma~\ref{lem:holder} ensures that the line-search for $x_1$ stops for 
	$	\hat L_1 \geq  \epsilon^\frac{s-2}{s} L^{\frac{2}{s}}.
	$
	Therefore we have $L_1 \leq 2\epsilon^\frac{s-2}{s} L^{\frac{2}{s}}$.
	For $k\geq 2$, during the line-search procedure, the parameter $\theta_k$ reads, denoting $a = \theta_{k-1}^2 L_{k-1}/\hat L_{k} $,
	$
	\theta_k(\hat L_k) = \frac{-a + \sqrt{a^ 2+ 4a}}{2} = \frac{2}{1+\sqrt{1+4/a}} = \frac{2}{1+ \sqrt{1+4\hat L_k/(\theta_{k-1}^2L_{k-1})}}
$.
	Using again Lemma~\ref{lem:holder}, the stopping criterion~\eqref{eq:suff_decrease} is then ensured if there exists $\hat L_k$ such that 
	\begin{equation}\label{eq:line_search_crit}
	\hat L_k \geq  \left(\theta_k(\hat L_k) \epsilon\right)^{\frac{s-2}{s}}L^{\frac{2}{s}} 
	= (2\epsilon)^\frac{s-2}{s} L^{\frac{2}{s}}\left(1+ \sqrt{1+4\hat L_k/(\theta_{k-1}^2L_{k-1})}\right)^{\frac{2-s}{s}},
	\end{equation}
	Denote $c: x \rightarrow x - \alpha(1+\sqrt{1+ \beta x})^ \gamma$ for $\alpha > 0, \beta> 0, 0 \leq \gamma \leq 1$. We have $\lim_{x\rightarrow +\infty} c(x) = +\infty$. Therefore there exists $\hat L_k>0$ satisfying~\eqref{eq:line_search_crit}. Moreover the line-search terminates with $L_k \leq 2 \left(\theta_k \epsilon\right)^{\frac{s-2}{s}}L^{\frac{2}{s}}$ as otherwise,
	$
	L_k/2 >  \left(\theta_k(L_k) \epsilon\right)^{\frac{s-2}{s}}L^{\frac{2}{s}} \geq \left(\theta_k(L_k/2) \epsilon\right)^{\frac{s-2}{s}}L^{\frac{2}{s}},
	$
	and the line-search would have stopped before.
\end{proof}
\begin{algorithm}[t]
	\begin{algorithmic}[1]\caption{Simplified Universal Fast Gradient Method $x = \mathcal{U}(x_0, t, \epsilon)$\label{algo:universal}}
		\State{{\bf Problem oracles:} Convex functions $\phi, g$, first-order oracles $(x,y, \lambda) \rightarrow \argmin_z y^Tz + g(z) + \lambda D_h(z;x)$ and $x \rightarrow \nabla \phi(x)$ with $\nabla \phi(x) \in \partial \phi(x)$ }
		\State{{\bf Inputs:} Initial point $x_0$, number of iterations $t$, target accuracy $\epsilon$, smoothness estimate $L_0$}
		\State{\textbf{Initialize:} $z_0 = x_0, \theta_1 = 1$}
		\For{$k=1,\ldots, t$}
		\State{Initialize line-search by $\hat L_k = L_{k-1}/2$}
		\Repeat
		\If{$k>1$}
		\State{Compute $\theta_k \geq 0$ s.t.
			\begin{equation}\label{eq:update_theta}
			\frac{1 -\theta_{k}}{\hat L_k\theta_{k}^2} = \frac{1}{L_{k-1}\theta_{k-1}^2}
			\end{equation}
			}
		\EndIf
		\State{Compute
		\begin{align*}
		y_{k-1} & = (1-\theta_k) x_{k-1} + \theta_k z_{k-1} \\
		z_{k} & = \argmin_z \ell_f(z; y_{k-1}) + \hat L_k \theta_{k} D_h(z; z_{k-1}) \\
		\tilde x_{k} & = (1-\theta_{k})x_{k-1} + \theta_k z_{k}
		\end{align*}}
		\IIf{$f(\tilde x_{k}) > \ell_f(\tilde x_{k}; y_{k-1}) + \frac{\hat L_k}{2}\|\tilde x_{k}-y_{k-1}\|^2 + \frac{\theta_k \epsilon}{2}$}{ $\hat L_k\gets 2\hat L_k$}\EndIIf
		\Until{\begin{equation}\label{eq:suff_decrease}
				f(\tilde x_{k}) \leq  \ell_f(\tilde x_{k}; y_{k-1}) + \frac{\hat L_k}{2}\|\tilde x_{k}-y_{k-1}\|^2 + \frac{\theta_k \epsilon}{2}
			\end{equation}}
			\State{Pick any $x_{k}$ such that 
				\begin{equation}\label{eq:monotonicity}
					f(x_{k}) \leq \min(f(\tilde x_{k}), f(x_{k-1}))
				\end{equation}}
			\State{Update  $L_{k} = \hat L_k$}
		\EndFor
		\State{{\bf Output:} $x_t$}
	\end{algorithmic}
\end{algorithm}

\subsection{Accelerated algorithm}
The accelerated algorithm is obtained as a special case of the universal fast gradient algorithm for $s=2$ and a choice of $\epsilon=0$, i.e., $\mathcal{A}(x_0, t) = \mathcal{U}(x_0, t, 0)$.  Its rate follows from the one given by the universal fast gradient method as recalled below.
\begin{corollary}
	Consider problem~\eqref{eq:pb_composite_algo} where $\phi$ satisfies~\eqref{eq:holder_smooth} with parameters $(2,L)$. Algorithm~\ref{algo:universal} with $\epsilon=0$, started at an initial point $x_0$ with an initial estimate $L_0$, outputs after $t$ iterations a point $x_t$ such that
	\[
	f(x_{t}) - f^ * \leq \frac{8 L }{t^2} D_h(x_0, X^*)\quad \mbox{and} \quad f(x_t) \leq f(x_0),
	\]
	where $D_h(x,X^*) = \min_{x^* \in X^*} D_h(x^*; x)$ with $X^* = \argmin_x f(x)$.
\end{corollary}

\subsection{Lemmas for proving convergence}
\begin{lemma}[{\citep[Lemma 2]{Nest15}}]\label{lem:holder}
	Let $\phi$ satisfying~\eqref{eq:holder_smooth}. Then for any $\delta>0$ and 
	$
	\hat L \geq \left(\frac{2-s}{s} \frac{1}{\delta}\right)^{\frac{2-s}{s}}L^{\frac{2}{s}},
	$
	or a fortiori any 
		$
	\hat L \geq \delta^\frac{s-2}{s} L^{\frac{2}{s}},
	$
	we have, for any $x,y \in \dom \phi$ and $\nabla \phi(x) \in \partial \phi(x)$,
	\begin{equation}
	\phi(y) \leq \phi(x) + \nabla \phi(x)^ \top (y-x) + \frac{\hat L}{2} \|x-y\|^2 + \frac{\delta}{2}. 
	\end{equation}
\end{lemma}
\begin{lemma}[{\citep[Property 1]{Tsen08}}]
	For any proper l.s.c. convex function $\psi$, and any $z \in\dom h$, denote
	$
	z_+ = \argmin_x \{ \psi(x) + D_h(x; z)\}.
$
	Then for any $x \in \dom h$, 
	\begin{equation}\label{eq:three_point}
	\psi(x) + D_h(x;z) \geq \psi(z_+) + D_h(z_+; z) + D_h(x; z_+).
	\end{equation}
\end{lemma}
\begin{lemma}
	Consider two sequences $(L_k)_{k\geq 1}, (\theta_k)_{k \geq 1}$ initialized by $L_1 > 0, \theta_1 = 1$ and satisfying
	\begin{equation}\label{eq:upper_L}
	\mbox{for $k\geq 1$} \qquad L_k \leq \frac{\alpha}{\theta_k^\beta}, \qquad \mbox{and, for $k\geq 2$,} \qquad \frac{1 -\theta_{k}}{L_k\theta_{k}^2} = \frac{1}{L_{k-1}\theta_{k-1}^2},
	\end{equation}
	for some $\alpha \geq 0, \beta \in [0,1]$, with $\theta_k \geq 0$.
	Then for any $k\geq 1$,
	\begin{equation}\label{eq:sum_theta}
	\frac{1}{L_k \theta^2_k} = \sum_{j=1}^ k \frac{1}{L_j \theta_j} \qquad \mbox{and} \qquad
	 L_k\theta^2_k \leq \frac{\alpha 2^{2-\beta}}{k^{2-\beta}} 
	\end{equation}
\end{lemma}
\begin{proof}
	The first property of~\eqref{eq:sum_theta} is true for $k=1$ since $\theta_1 = 1$.
	The definition of $\theta_k$ for $k\geq2$ reads then
	$
	\frac{1}{L_k \theta_k^2} = \frac{1}{L_{k-1} \theta_{k-1}^2} + \frac{1}{L_k \theta_k}
	$
	which shows the first property~\eqref{eq:sum_theta} by induction.
	
	For $k\geq 1$, denote $a_k = \frac{1}{L_k \theta_k}$ and $A_k = \sum_{j=1}^k a_j$, such that 
	$A_k =\frac{1}{L_k \theta_k^2} =L_k a_k^2$. We have from~\eqref{eq:upper_L}, $L_k^{1-\beta} \leq \alpha a_k^{\beta}$.
	Therefore $A_k^{1-\beta} \leq \alpha a_k^{2-\beta}$ which gives
$
	     A_k^{\frac{1-\beta}{2-\beta}} \leq \alpha^{\frac{1}{2-\beta}} a_k.
$
	Denote $\gamma = \frac{1}{2-\beta}$ and $A_0=0$. Since $\gamma \geq 1/2$ and $A_{k} \geq A_{k-1}$, we have for any $k\geq 1$,
	\begin{align*}
	    A_{k}^{\gamma} - A_{k-1}^\gamma \geq \frac{A_{k} - A_{k-1}}{A_{k}^{1-\gamma} + A_{k-1}^{1-\gamma}} \geq \frac{A_{k} - A_{k-1}}{2A_{k}^{1-\gamma}} \geq \frac{1}{2\alpha^{\frac{1}{2-\beta}}}.
	\end{align*}
	Therefore we conclude that 
	$
	    A_k \geq \frac{k^{\frac{1}{\gamma}}}{2^{\frac{1}{\gamma}} \alpha} = \frac{k^{2-\beta}}{2^{2-\beta} \alpha}.
	$
\end{proof}

\subsection{Smoothing non-smooth problems}
We present here the smoothing algorithm used in Section~\ref{sec:compobreg}. We recall the problem
\begin{equation}
\mbox{minimize} \quad f(x) \triangleq \phi(Ax) + g(x)
\end{equation}
where $A \in \reals^{m\times n}$, $g$ is a simple convex function and $\phi$ is a non-smooth convex function whose inf-convolution with some smooth convex function $\psi$  can be computed analytically, i.e. one has access for any $\mu>0$ to 
\[
\phi_{\mu \psi^\star}(x) = \inf_y \left\{\phi(y) + \mu\psi\left(\frac{x-y}{\mu}\right) \right\} =\sup_{u} \left\{u^\top x - \phi^\star(u) - \mu \psi^\star(u)\right\},
\]
where for a function $f$ we denote by $f^\star$ its convex conjugate, see \citep{Beck12} for a detailed exposition. For $\psi^\star$ 1-strongly convex  w.r.t. a given norm $\|\cdot\|_\beta$ (i.e. $\psi$ 1-smooth w.rt. $\|\cdot\|_\beta^*$), the function $\phi_{\mu \psi^\star}$ is $1/\mu$ smooth w.r.t. the norm $\|\cdot\|_\beta^*$. Moreover it approximates $\phi$ as (see e.g. \citep[Proposition 41]{pillutla2019smoother})
\begin{equation}\label{eq:smoothing_approx}
\mu \inf_{u \in \dom \phi^\star}	\psi^\star(u) \leq \phi(x) - \phi_{\mu \psi^\star}(x) \leq  \mu \sup_{u \in \dom \phi^\star}	\psi^\star(u) \quad \mbox{ for any $x\in \dom \phi$}.
\end{equation}
The smoothed objective is a composite objective as in~\eqref{eq:pb_composite_algo}, i.e.,
\begin{equation}\label{eq:smoothed_pb}
f_{\mu \psi^\star}(x) = \phi_{\mu \psi^\star}(Ax) + g(x)
\end{equation}
where, denoting $\|A\|_{\alpha, \beta} = \sup_{ \|x\|_\alpha \leq 1, \|u\|_\beta \leq 1} u^\top A x$, we have that $x \rightarrow \phi_{\mu \psi^\star}(Ax)$ is $\|A\|_{\alpha, \beta}^2/\mu$ smooth w.r.t. a norm $\|\cdot\|_\alpha$ (see e.g. \citep[Lemma 42]{pillutla2019smoother}). We can then apply the accelerated algorithm on~\eqref{eq:smoothed_pb} with a potential function $h$ strongly convex w.r.t. the norm $\|\cdot\|_\alpha$, assuming that $g$ is simple such that we have access to oracles of the form~\eqref{eq:prox_g}. 
Precisely, given an initial point $x_0$ and a target accuracy $\epsilon$, by applying the accelerated algorithm on~\eqref{eq:smoothed_pb} with $\mu = \epsilon/(2D)$, where $D =  \sup_{u \in \dom \phi^*}	\psi^*(u) - \inf_{u \in \dom \phi^*}\psi^*(u)$, we get after $t$ iterations a point $\tilde x$ such that for $x^* \in \argmin f$,
\[
f(\tilde x) - f^* \stackrel{\eqref{eq:smoothing_approx}}{\leq} f_{\mu \psi^\star}(\tilde x) - f_{\mu \psi^\star}(x^*) + D\mu \stackrel{(i)}{\leq} \frac{\epsilon}{2} + \frac{16D\|A\|_{\alpha, \beta}^2 }{\epsilon t^2} D_h(x^*; x_0).
\]
where in $(i)$ we use the definition of $\mu$ and the convergence bound~\eqref{eq:bound_univ_x} of the accelerated algorithm ($\epsilon=0$) applied on $x=x^*$.
We denote then $x = \mathcal{S}(x_0, \epsilon, t)$ any point such that $f(x) \leq \min\{f(\tilde x), f(x_0)\}$ such that it satisfies both the rate above and belongs to the initial sub-level set. The bound presented in equation~\eqref{eq:smoothing_rate} is obtained by taking $x^*\in \argmin_{x\in X^*} D_h(x; x_0)$ and defining $c=32$ and $L_{\psi^*, A} = \sqrt{D}\|A\|_{\alpha, \beta}$.
	
	\section{Rounding issues} \label{sec:total_number_approx}
	We presented convergence bounds for real sequences of iterate counts $(t_k)_{k=1}^\infty$ but in practice these are integer sequences. The following Lemma details the convergence of our schemes for an approximate choice $\tilde{t}_k= \lceil t_k \rceil$.

\begin{lemma}
	Let $x_k$ be a sequence whose $k\textsuperscript{th}$ iterate is generated from previous one by an algorithm that needs $t_k$ iterations and denote $N = \sum_{k=1}^R t_k$ the total number of iterations to output a point $\hat{x} = x_R$.
	Suppose setting 
	$
	t_k = \lceil Ce^{\alpha k} \rceil, \quad k=1,\ldots,R
	$
	for some $C>0$ and $\alpha\geq 0$  ensures that objective values $f(x_k)$ converge linearly, i.e.
	\BEQ \label{eq:conv_obj_approx}
	f\left(x_k\right) -f^* \leq \nu e^{-\gamma k},
	\EEQ
	for all $k\geq 0$ with $\nu\geq 0$ and $\gamma \geq 0$. Then precision at the output is given by,
	\[
	f(\hat{x}) -f^* \leq \nu \exp(-\gamma N/(C+1)),\quad \mbox{when $\alpha=0$,}
	\]
	and, denoting $N' = N - \frac{\log\left( (e^\alpha-1)e^{-\alpha}C^{-1}N +1\right)}{\alpha}$,
	\[
	f(\hat{x}) -f^* \leq \frac{\nu}{\left(\alpha e^{-\alpha} C^{-1}N'+1\right)^\frac{\gamma}{\alpha}},\quad \mbox{when $\alpha>0$.} 
	\]
\end{lemma}
\begin{proof}
	At the $R\textsuperscript{th}$ point generated, $N = \sum_{k=1}^R t_k$. If $t_k = \lceil C \rceil$, define $\epsilon = \lceil C \rceil -C$ such that $0 \leq \epsilon <1$. Then $ N = R(C+\epsilon)$, injecting it in \eqref{eq:conv_obj_approx} at the $R\textsuperscript{th}$ point,
	\[
	f(\hat{x}) -f^* \leq \nu e^{-\gamma \frac{N}{C+\epsilon}} \leq \nu e^{-\gamma \frac{N}{C+1}}.
	\]
	
	Now, if $t_k = \lceil Ce ^{\alpha k} \rceil$, define $\epsilon_k = \lceil Ce ^{\alpha k} \rceil -C e ^{\alpha k}$, such that $0 \leq \epsilon_k <1$. On one hand
	$
	N \geq \sum_{k=1}^R C e ^{\alpha k},
	$
	such that 
	$
	R \leq \frac{\log\left( (e^\alpha-1)e^{-\alpha}C^{-1}N +1\right)}{\alpha}.
	$
	On the other hand, 
	\BEAS
	N = \sum_{k=1}^R t_k  =  \frac{Ce^\alpha}{e^\alpha-1}(e^{\alpha R}-1) + \sum_{k=1}^R \epsilon_k 
	& \leq & \frac{Ce^\alpha}{e^\alpha-1}(e^{\alpha R}-1) + R \\
	& \leq &  \frac{Ce^\alpha}{e^\alpha-1}(e^{\alpha R}-1) + \frac{\log\left( (e^\alpha-1)e^{-\alpha}C^{-1}N +1\right)}{\alpha},
	\EEAS
	such that
	$
	R \geq \frac{\log\left( \alpha e^{-\alpha}C^{-1}N' +1\right)}{\alpha}.
	$
	Injecting it in \eqref{eq:conv_obj_approx} at the $R\textsuperscript{th}$ point the result follows.
\end{proof}
	
	\section{Grid-search for universal restart schemes}\label{sec:universal_grid_search}
	We briefly explain why a grid-search on the parameters for the general case $s\in[1, 2]$ does not provide near-optimal bounds.
Consider general restart schemes as presented in Algorithm~\ref{algo:scheduled_gen} for a function $f \in \mathcal{H}_{s,L} \cap \mathcal{L}_{r, \mu}(K)$ with $K$ the initial sublevel set of $f$ at a given $x_0$. Assume that the decreasing factor is $\gamma$ and the schedules have the form 
$
t_k = Ce^{\alpha k}
$
such that 
\[
t_k \geq e^{\gamma\frac{1-\tau}{\rho}}(c \kappa)^{\frac{s}{2\rho}} \epsilon_0^{-\frac{\tau}{\rho}}  e^{ \frac{\gamma \tau}{\rho} k },
\]
which is $C \geq  e^{\gamma\frac{1-\tau}{\rho}}(c \kappa)^{\frac{s}{2\rho}} \epsilon_0^{-\frac{\tau}{\rho}}$ and $\alpha \geq \frac{\gamma \tau}{\rho}$.
Consider the case $\tau >0$. Then following the proof of Proposition~\ref{th:sched_gen}, we get that 
$
f(x_R) - f^* \leq \gamma^k \epsilon_0
$
and applying Lemma~\ref{lem:total_number} we obtain a convergence rate
\[
f(x_R) - f^* \leq \frac{\epsilon_0}{\left(\alpha e^{-\alpha} C^{-1}N +1\right)^\frac{\gamma}{\alpha}},
\]
where $N$ is the total number of iterations.
For this rate to be optimal or nearly optimal we need $\frac{\gamma}{\alpha} = \frac{\rho}{\tau}$. Any grid search on this ratio will then suffer from a constant factor such that we won't get a rate of the form $N^{-\frac{\rho}{\tau}}$ except if we know $r$ and $s$ which gives us $\rho/\tau$.
	
\end{document}